\documentclass[a4paper,12pt]{article}
\usepackage[left=3cm, right=3cm, top=2.5cm, bottom=2.5cm, bindingoffset=0cm]{geometry}
\usepackage[shortlabels]{enumitem}
\usepackage[title]{appendix}
\usepackage{mathtools}
\usepackage[all, cmtip]{xy}
\usepackage[T2A]{fontenc}			
\usepackage[utf8]{inputenc}	
\usepackage{subfigure}
\usepackage{comment} 
\usepackage{commath}
\usepackage[final]{epsfig}
\usepackage{graphics}
\usepackage{amsmath}
\usepackage{amsfonts}
\usepackage{amsthm}
\usepackage{float}
\usepackage{latexsym}
\usepackage{amssymb, mathabx}
\usepackage{pifont}
\usepackage{authblk}
\usepackage{epstopdf, setspace}
\usepackage{multicol,multirow, enumitem}
\usepackage{hyperref, mathtools}
\newlist{longenum}{enumerate}{5}
\setlist[longenum, 1]{label=\roman*)}
\setlist[longenum, 2]{label=\alph*)}
\usepackage{array}
\usepackage{wrapfig}
\usepackage{calrsfs}
\DeclareMathAlphabet{\pazocal}{OMS}{zplm}{m}{n}
\usepackage{tikz}
 \usepackage{tikz-cd}
\usetikzlibrary{positioning}
\usetikzlibrary{decorations.text}
\usetikzlibrary{decorations.pathmorphing}
\usetikzlibrary{decorations.markings}
\usetikzlibrary{patterns}
\usepackage{pgfplots}
\usetikzlibrary{arrows} 

\usetikzlibrary{backgrounds,fit,petri}

\tikzset{->-/.style={decoration={
markings,
mark=at position .7 with {\arrow{>}}},postaction={decorate}}}
\tikzstyle{vertex} = [coordinate]

\newtheorem{lemma}{Lemma}[section]
\newtheorem{proposition}[lemma]{Proposition}
\newtheorem{remark-definition}[lemma]{Remark-Definition}
\newtheorem{theorem}[lemma]{Theorem}

\newtheorem{corollary}[lemma]{Corollary}

\newtheorem{proposition-conjecture}[lemma]{Proposition-conjecture}

\theoremstyle{definition}
\newtheorem{example}[lemma]{Example}

\newtheorem{definition}[lemma]{Definition}
\newtheorem{remark}[lemma]{Remark}

\newcommand{\R}{{\mathbb R}}
\newcommand{\Z}{{\mathbb Z}}
\newcommand{\RP}{{\mathbb {RP}}}

\newcommand{\vect}{\mathrm{Vect}}
\newcommand{\Vect}{\mathrm{Vect}}

\newcommand{\diff}[1]{\mathrm{d}  #1}

\newcommand{\Hom}{{H}}

\newcommand{\ad}{\mathrm{ad}}
\newcommand{\Ad}{\mathrm{Ad}}

\newcommand{\id}{\mathrm{Id}}

\newcommand{\Diffeo}{\mathrm{Diff}}

\newcommand{\oneform}{\alpha}
\newcommand{\oneformtwo}{\beta}

\newcommand{\circulation}{ \mathcal C}

\newcommand{\Diff}{\mathrm{Diff}}

\newcommand{\curl}{\mathrm{curl}}

\newcounter{ai}

\newcounter{ik}

\mathtoolsset{showonlyrefs}

\title{
Coadjoint orbits of area-preserving diffeomorphisms of non-orientable surfaces}
\author{
Anton Izosimov\thanks{Department of Mathematics, University of Arizona, e-mail: \tt{izosimov@math.arizona.edu}}, Boris Khesin\thanks{Department of Mathematics, University of Toronto
e-mail: {\tt khesin@math.toronto.edu}
}\,, and Ilia Kirillov\thanks{Department of Mathematics, University of Toronto
e-mail: {\tt kirillov@math.utoronto.ca}
}
}

\date{}
\pgfplotsset{compat=1.14} 
\begin{document}
\maketitle
\begin{abstract}
We give a classification of generic coadjoint orbits for the group of area-preserving diffeomorphisms of a closed non-orientable surface. This completes V.\,Arnold's program of studying invariants of incompressible fluids in 2D.
As an auxiliary problem, we also classify simple Morse pseudo-functions on non-orientable surfaces up to area-preserving diffeomorphisms.
\end{abstract}

\tableofcontents
\section{Introduction}
\label{section:intro}

Non-orientable manifolds do not carry non-vanishing volume forms but allow {densities} (also called pseudo-forms). 
Densities can be thought of as non-vanishing top-degree forms  whose sign changes after returning to the same point along an orientation-reversing loop. Fix a density $\rho$ on a non-orientable manifold $N$ and consider the infinite-dimensional group $\Diff_\rho(N)$ of {measure-preserving diffeomorphisms} of $N$. In the present paper we study the group~$\Diff_\rho(N)$ in the case when $N$ is a closed non-orientable surface. Our first main result is a classification of generic pseudo-functions on such surfaces with respect to the action of~$\Diff_\rho(N)$. The second result is a classification of generic coadjoint orbits of~$\Diff_\rho(N)$. 
We also comment on Casimir functions (i.e. invariants of the coadjoint action) and similar classifications for the case of non-orientable surfaces with boundary.

The motivation for our study comes from hydrodynamics. Two-dimensional ideal fluid dynamics is a good approximation for many real world processes (for instance, the earth atmosphere is well-approximated by 2D fluid equations). Hydrodynamics on non-orientable surfaces can also be observed in nature: for instance, the motion of a soap film on a wire takes place on a minimal surface, which, for a suitable wire shape, is a M\"obius band. Fluids on non-orientable surfaces were considered e.g. in~\cite{balabanova2022hamiltonian1, balabanova2022hamiltonian2, vanneste2021vortex}. 
Beyond hydrodynamics, coadjoint orbits of area-preserving diffeomorphisms also arise in general relativity \cite{donnelly2021gravitational, penna2020sdiff} in the context of correspondence between coadjoint orbits and representations suggested by the orbit method.

 \begin{figure}[b]
\centerline{
\begin{tikzpicture}[scale = 0.9]
\draw (0,1) .. controls (0,1.5) and (0.25,2) .. (0.5,2)
(0.5,2) ..controls (0.75,2) and (0.75,1.5) .. (1,1.5)
(2,2.5) ..controls (1.5,2.5) and (1.5,1.5) .. (1,1.5)
(2.2,1) .. controls (2.4,1.75) and (2.3,2.5) .. (2,2.5)
(1,-0.5) ..controls (0.5,-0.5) and (0,0.5) .. (0,1)
(1,-0.5) ..controls (2,-0.5) and (2,0.5) .. (2.2,1);
    \draw   (1.2,0.15) arc (260:100:0.5cm);
    \draw   (1,0.25) arc (-80:80:0.4cm);
\node [vertex] at (5,-0.5) (nodeA) {};
\node [vertex] at (5,0.2) (nodeB) {};
    \draw  [->-] (nodeA) -- (nodeB);
\node [vertex] at (5,1) (nodeC) {};
    \draw  [->-] (nodeB) .. controls +(-0.3,+0.4) .. (nodeC);
        \draw  [->-] (nodeB) .. controls +(0.3,+0.4) .. (nodeC);
            \node at (5.5,1.2) (nodeZ) {$\Gamma_F$};
\node [vertex] at (5,1.5) (nodeD) {};
    \draw  [->-] (nodeC) -- (nodeD);
\node [vertex] at (4.5,2) (nodeE) {};
    \draw  [->-] (nodeD) -- (nodeE);
\node [vertex] at (5.5,2.5) (nodeF) {};
    \draw  [->-] (nodeD) -- (nodeF);
\fill (nodeA) circle [radius=1.5pt];
\fill (nodeB) circle [radius=1.5pt];
\fill (nodeC) circle [radius=1.5pt];
\fill (nodeD) circle [radius=1.5pt];
\fill (nodeE) circle [radius=1.5pt];
\fill (nodeF) circle [radius=1.5pt];
\node [vertex] at (1.02,-0.5) (nodeAi) {};
\node [vertex] at (1.02,1.5) (nodeDi) {};
\node [vertex] at (0.5,2) (nodeEi) {};
\node [vertex] at (2,2.5) (nodeFi) {};
\fill (nodeAi) circle [radius=1.5pt];
\fill (nodeDi) circle [radius=1.5pt];
\fill (nodeEi) circle [radius=1.5pt];
\fill (nodeFi) circle [radius=1.5pt];
    \draw  [->] (-0.5,-0.5) -- (-0.5,2.5);
      \draw  [->, densely dashed] (3,1) -- (4,1);
    \node at (-0.8,1) (nodeA) {$F$};
        \node at (2.2,-0.1) (nodeA) {$M$};
\end{tikzpicture}
}
\caption{Reeb graph for a height function with two maxima on a torus.}\label{torusInt}
\end{figure}
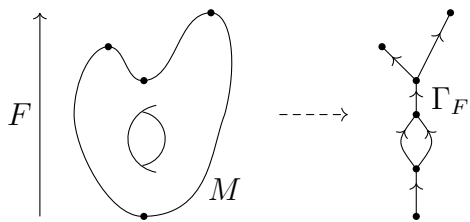

The presented classifications complete V.\,Arnold's program of studying coadjoint orbits of area-preserving diffeomorphisms initiated in  \cite{arnold1999topological} in view of applications to hydrodynamic stability. Recall that the corresponding problem for orientable surfaces with and without boundary was recently solved in \cite{izosimov2016coadjoint, izosimov2016characterization, kirillov2023classification}. At heart of those classifications is the description of orbits of functions under the action of area-preserving diffeomorphisms in terms of  measured Reeb graphs. The graph $\Gamma_F$, called the Reeb graph, is the set of connected components of level sets of a  function~$F$ on an orientable surface $M$. Critical points of $F$ correspond to the vertices of the graph~$\Gamma_F$, see Figure \ref{torusInt}.  
This graph comes with a natural parametrization by the values of $F$. For a surface equipped with an area form $\mu$, the measure $\mu$ induces a measure on the graph, which satisfies certain properties discussed below. 

A natural way to describe objects on a non-orientable manifold $N$ is to lift them to the double cover $\widetilde N$, which is an oriented manifold. This double cover comes with  a
fixed-point-free orientation-reversing involution $I: \widetilde N\to \widetilde N$ interchanging the
points in each fiber of the natural projection $\widetilde N\to N$. Pseudo-functions $F$ on a non-orientable manifold  $N$ are functions on its double cover $\widetilde N$ anti-invariant under the involution: 
$F\circ I=-F$. One can define simple Morse pseudo-functions on $N$ in a natural way: their lifts to $\widetilde N$ have to be Morse with distinct critical values. Our first result is the density of such pseudo-functions among all:

\begin{theorem} {\bf (=Theorem \ref{theorem_simple_pseudo})}
Simple Morse pseudo-functions on a compact non-orientable manifold form an open and dense subset in the space of all smooth pseudo-functions in $C^2$-topology.
\end{theorem}

Our next result is a classification of simple Morse pseudo-functions in 2D. Let~$N$ be a closed (i.e. compact and without boundary) non-orientable surface equipped with density $\rho$. 
It turns out that invariants of the $\Diff_\rho(N)$-action on pseudo-functions 
are given by measured Reeb graphs of their lifts to the orientation double cover $\widetilde N$, equipped with an involution.

\begin{theorem}{\bf (=Theorem \ref{classification_of_pseudo-functions_symplectic})}
Let $N$ be a closed connected non-orientable 2D surface equipped with a density $\rho$. Then there is a one-to-one correspondence between simple Morse pseudo-functions on $N$, considered up to area-preserving diffeomorphisms, and isomorphism classes of measured Reeb graphs with involution compatible with $(N,\rho)$.
\end{theorem}

\begin{example}
 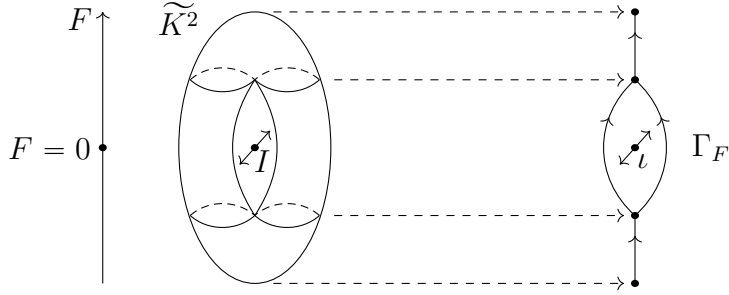
\begin{figure}[t]
\centerline{
\begin{tikzpicture}[yscale = 0.9]
    \draw (2,2) ellipse (1cm and 2cm);
   \draw   (1.15,1.0) to [bend right = 45] (2,1.);
   \draw   (2.85,1.0) to [bend left = 45] (2,1.);
   \draw   [densely dashed] (1.15,1.0) to [bend left = 45] (2,1.);
   \draw   [densely dashed] (2.85,1.0) to [bend right = 45] (2,1.);
      \draw   (1.15,3.0) to [bend right = 45] (2,3.);
   \draw   (2.85,3.0) to [bend left = 45] (2,3.);
   \draw   [densely dashed] (1.15,3.0) to [bend left = 45] (2,3.);
   \draw   [densely dashed] (2.85,3.0) to [bend right = 45] (2,3.);
   \draw (2,1.) to [bend left = 30] (2,3.);
     \draw (2,1.) to [bend right = 30] (2,3.);
    \draw  [->] (0,0) -- (0,4);
    \node at (-0.3,3.9) (nodeA) {$F$};
    \node at (-0.7,2) (nodeAA) {$F=0$};
    \fill (0,2) circle [radius=1.5pt];
    \node at (1,3.9) (nodeG) {$\widetilde{K^2}$};
    \node at (8,2) (nodeB) {$\Gamma_F$};
    \node [vertex] at (7,0) (nodeC) {};
    \node [vertex]  at (7,1) (nodeD) {};
    \node [vertex] at (7,3) (nodeE) {};
    \node [vertex]  at (7,4) (nodeF) {};
    \node  at (2.1,4) (nodeG) {};
    \node  at (7,4) (nodeFdouble) {};
    \node  at (7,0) (nodeCdouble) {};
    \node  at (7,1) (nodeDdouble) {};
    \node at (7,3) (nodeEdouble) {};
    \node  at (2.1,0) (nodeH) {};
    \node  at (2.9,1) (nodeI) {};
    \node  at (2.9,3) (nodeJ) {};
    \draw  [->-] (nodeC) -- (nodeD);
    \fill (nodeC) circle [radius=1.5pt];
    \fill (nodeD) circle [radius=1.5pt];
    \fill (nodeE) circle [radius=1.5pt];
    \fill (nodeF) circle [radius=1.5pt];
    \draw  [->-] (nodeD)  to [bend left = 45] (nodeE);
    \draw  [->-] (nodeD)  to [bend right = 45] (nodeE);
    \draw  [->-] (nodeE) -- (nodeF);
    \draw  [dashed, ->] (nodeG) -- (nodeFdouble);
    \draw  [dashed, ->] (nodeH) -- (nodeCdouble);
    \draw  [dashed, ->] (nodeI) -- (nodeDdouble);
    \draw  [dashed, ->] (nodeJ) -- (nodeEdouble);
    \node at (7.1,1.8) (nodeAAA) {$\iota$};
    \draw [ <->] (1.8,1.75) --(2.2,2.25);
       \draw [ <->]  
    (6.8,1.75)-- (7.2,2.25);
    \node at (2.1,1.8) (nodeAAA) {$I$};
    \fill (2,2) circle [radius=1.5pt];
    \fill (7,2) circle [radius=1.5pt];
\end{tikzpicture}
}
\caption{Reeb graph for a height function on a torus  inducing a pseudo-function on the Klein bottle $K^2$. The involution $I$ on the torus and $\iota$ on the graph are central symmetries.}\label{torus}
\end{figure}
    The height function $F$ on a torus $T^2=\widetilde{K^2}$ is odd with respect to the central symmetry $I$ and hence induces a pseudo-function on the Klein bottle $K^2 = T^2\,/\,I$, see Figure \ref{torus}. The measured Reeb graph~$\Gamma_F$ with an involution $\iota$ is a complete invariant of the corresponding pseudo-function on the Klein bottle.
\end{example}

The classification of coadjoint orbits of the group $\Diff_\rho(N)$ of area-preserving diffeomorphisms of a non-orientable surface $N$ requires a more subtle set of data than the measured Reeb graph with an involution. Namely, elements of the regular dual space $\Vect_\rho^*(N) $ to the Lie algebra $\Vect_\rho(N)$ are 1-form cosets $[\alpha]\in  \Omega^1(N)\,/\,\diff\Omega^0(N)$. One associates to such a coset the vorticity pseudo-function $\curl[\alpha] := {\diff \alpha}/{\rho}$. This way the classification of 
such cosets with respect to area-preserving diffeomorphisms can be seen as a refinement of the pseudo-function classification: one needs to augment the measured Reeb graph of $F=\curl[\alpha]$ by additional information, carried by the so-called circulation function described below.
This allows one to formulate the full classification of generic coadjoint orbits in terms of  
circulation graphs.

\begin{theorem}{\bf (=Theorem \ref{thm:sdiffN})} 
Let $N$ be a closed connected non-orientable surface equipped with a density $\rho.$ Then simple Morse coadjoint orbits of $\Diff_\rho(N)$ are in one-to-one correspondence with isomorphism classes of circulation graphs compatible with $(N,\rho)$. 
\end{theorem}

\begin{corollary}{\bf (=Corollary \ref{cor:dim})} \label{cor:dim0}
Let $N$ be a closed connected non-orientable surface equipped with a density $\rho$.
Then the space of coadjoint orbits of the group $\Diff_\rho(N)$ corresponding to the same measured Reeb graph $\Gamma$ with involution $\iota$ is an affine space of dimension
   \begin{equation}
     d =  \frac{1}{2}(  \#\mathrm{Fix}(\iota) +  b_1(N) - 1),     
   \end{equation}
    where $\#\mathrm{Fix}(\iota)$ is the number of fixed points of the involution $\iota$, and $b_1(N) = \dim \Hom^{}_1(N; \R)$ is the first Betti number of $N$. In particular,
   \begin{equation}
      \frac{1}{2}(  b_1(N) - 1)  \leq d \leq b_1(N).
    \end{equation}
\end{corollary} 

\begin{remark}
    Here we encounter a completely new phenomenon, not observed for orientable surfaces. Namely,
    for an orientable surface $M$ the corresponding dimension $d$ is always $ \frac{1}{2}   b_1(M)$, i.e. the genus of $M$. On the other hand, for non-orientable surfaces the dimension of the space of coadjoint orbits for a given vorticity
    is determined not only by the topology of the surface, but also by more subtle information about the involution action on the vorticity  Reeb graph.     
\end{remark}

\begin{example}[ for details see Example \ref{ex:Klein3p}]\label{ex:Klein3}
 \begin{figure}[t]
\centerline{
\begin{tikzpicture}[xscale = 0.75, yscale = 0.6]
    \node [vertex] at (7,0) (nodeC) {};
    \node [vertex]  at (7,1) (nodeD) {};
    \node [vertex] at (7,3) (nodeE) {};
    \node [vertex]  at (7,4) (nodeF) {};
    \node  at (7,4) (nodeFdouble) {};
    \node  at (7,0) (nodeCdouble) {};
    \node  at (7,1) (nodeDdouble) {};
    \node at (7,3) (nodeEdouble) {};
    \draw  [->-] (nodeC) -- (nodeD);
    \fill (nodeC) circle [radius=1.5pt];
    \fill (nodeD) circle [radius=1.5pt];
    \fill (nodeE) circle [radius=1.5pt];
    \fill (nodeF) circle [radius=1.5pt];
    \draw  [->-] (nodeD)  to [bend left = 45] (nodeE);
    \draw  [->-] (nodeD)  to [bend right = 45] (nodeE);
    \draw  [->-] (nodeE) -- (nodeF);
       \draw [ <->] 
    (8,1.5) --
   (8,2.5) node[right] {$\iota$};;
   \draw [dashed] (5,2) -- (9,2);
\end{tikzpicture}
}
\caption{A graph involution $\iota$ with two fixed points. It is given by symmetry with respect to the dashed line.}\label{torusgraph2}
\end{figure}
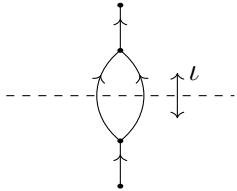

Consider the height function on a vertically standing torus  shown in Figure~\ref{torus}, interpreted as a pseudo-function on the Klein bottle $K^2.$ Since $b_1(K^2) = 1$ and $\#\mathrm{Fix}(\iota) = 0$, the space of coadjoint obits corresponding to the given graph is $0$-dimensional, i.e. the graph completely determines the orbit.

On the other hand, consider the height function on the torus lying on a  slightly inclined table. It is again odd with respect to the central symmetry and hence defines  a pseudo-function on $K^2$.
The corresponding Reeb graph is shown in Figure \ref{torusgraph2}. Here one has $\#\mathrm{Fix}(\iota) = 2$, so the space of coadjoint orbits of $\Diff_\rho(K^2)$ corresponding to such a function is $1$-dimensional. 
Note that $b_1(K^2)=1$, so $0$ and $1$ are the only possible dimensions $d$ of the orbit space for the Klein bottle.
\end{example}

This paper is organized as follows. In Section~\ref{section:geometry} we discuss certain notions of differential geometry relevant for non-orientable manifolds. 
In Section~\ref{section:func} we present a classification of simple Morse pseudo-functions on non-orientable surfaces up to an area-preserving transformation. In Section~\ref{section:orbits} we obtain a classification of generic coadjoint orbits for the group of area-preserving diffeomorphisms of a non-orientable surface. 

In Appendix~\ref{sect:hydro}, we present the hydrodynamic motivation for these classifications.
{ In Appendix~\ref{section:vort}, we describe singular coadjoint orbits  corresponding to vortex membranes in the non-orientable setting. Namely, it turns out that one can define an analog of the Marsden–Weinstein symplectic structure on co-oriented submanifolds of codimension 2. Furthermore, the binormal equation 
governing the evolution of vortex membranes also makes sense on non-orientable manifolds.}

\paragraph{Acknowledgements.} A.I. was supported by NSF grant DMS-2008021. B.K. was partially supported by an NSERC Discovery Grant. The authors are grateful to Klas Modin and Gleb Smirnov for fruitful discussions.

\section{Geometry of non-orientable manifolds}\label{section:geometry}
\subsection{Orientation double cover and orientation bundle}
Let $N$ be a non-orientable manifold. We define its \emph{orientation double cover} $\widetilde N$ 
as the space of pairs $(x, \mathcal O)$, where $x \in N$, and $\mathcal O$ is an orientation of the tangent space~$T_x N$. Clearly, $\widetilde N$ is an orientable manifold. Moreover, defined in this way, the orientation double cover $\widetilde N$ is \emph{canonically oriented}: the orientation of $T_{(x,\mathcal O)}\widetilde N$ is defined by pulling back the orientation $\mathcal O$ from~$T_x N$.

Let $N$ be a non-orientable manifold, and $\widetilde N$ be its double cover. Then there is a fixed-point-free orientation-reversing involution $I\colon \widetilde{N}\to \widetilde{N}$ which interchanges the points in each fiber of the projection $ \widetilde N \to N$.  Conversely, let $M$ be an {oriented} \emph{connected} manifold equipped with a fixed-point-free orientation-reversing involution~$I$. 
Then the quotient space $M\,/\,I$ is a non-orientable manifold whose orientation double cover is canonically diffeomorphic to $M$. 
Thus, one can go back and forth between connected non-orientable manifolds and connected oriented manifolds equipped with a fixed-point-free orientation-reversing involution. We will be using both models interchangeably throughout the paper. { Note that while, in principle, everything can be done on the double cover, in some cases we found it more convenient to directly work with the non-orientable manifold. }

Closely related to the orientation double cover is the concept of \emph{orientation bundle}. Viewing the orientation double cover $\widetilde N \to N$ as a principal $\Z_2$-bundle, define the {orientation bundle} as the associated line bundle. In other words, given an atlas on $N$, the orientation bundle is given by transition functions $\mathrm{sign}(J(\phi_{\alpha\beta}))$ where $\phi_{\alpha\beta}$ are transition maps between charts, and $J$ is the Jacobian determinant. We denote the orientation bundle of $N$ by $o(N)$.
\subsection{Differential forms of even and odd type}
\label{section:dif_forms}

The language of even and odd differential forms was introduced by de Rham \cite{de2012differentiable}. Let~$N$ be a non-orientable manifold.
\emph{Differential $k$-forms of even type} are just regular differential $k$-forms,  i.e. sections of $\bigwedge^k T^*N.$ 
\emph{Differential $k$-forms of odd type}, also called \emph{pseudo-forms}, are regular differential forms twisted by the orientation bundle~$o(N)$, i.e. sections of $ \bigwedge^k T^*N \,\otimes \, o(N)$. In more concrete terms, odd forms can be defined as follows:
\begin{definition}
An \emph{odd $k$-form} $\alpha$ on a vector space $V$ assigns to each orientation ${\mathcal O}$ of $V$ an exterior $k$-form $\alpha_{{\mathcal O}}$ such that if the orientation is reversed the exterior form is replaced by its negative:
$\alpha_{-{\mathcal O}}=-\alpha_{\mathcal O}.$
An odd differential $k$-form on a manifold~$N$ assigns an odd $k$-form $\alpha$ to each tangent space $T_x N$ in a smooth fashion.
Non-vanishing odd forms of top degree are called \emph{densities}. Odd $0$-forms are called \emph{pseudo-functions}.

Denote the space of even (i.e. usual) $k$-forms on a manifold $N$ by $ \Omega^k(N)$ and odd $k$-forms by $\widetilde \Omega^k(N)$.
\end{definition}

\begin{example}
Let $(N, g)$ be a Riemannian manifold of dimension $n$ (orientable or not). Then we have an associated Riemannian density 
$
\rho_g \in \widetilde \Omega^{n}(N),
$
defined as the unique odd $n$-form
that assigns to an orientation $\mathcal O$ of the tangent space $T_p N$ and an orthonormal basis in $T_p N$ the value $+1$ if the orientation of the basis agrees with $\mathcal O$ and $-1$ otherwise. In terms of local coordinates the density $\rho_g$ has the following expression:
$$\rho_g = \sqrt{\det g}\, \diff x^1\wedge \dots \wedge \diff x^n.$$
(Note that locally we can always write odd forms as usual differential forms, since local coordinates define a local trivialization of the orientation bundle.)
\end{example}

 Assuming $N$ to be non-orientable, there is a one-to-one correspondence between even (respectively, odd) differential forms on $N$ and those differential forms 
on the orientation double cover $\widetilde{ N}$ that are even (respectively, odd) with respect the involution $I\colon \widetilde{ N}\to \widetilde{ N}$:
\begin{gather}\Omega^{k}(N)\simeq \Omega^{k}_{even}(\widetilde N) = \{\omega \in \Omega^k(\widetilde{ N})\colon I^*\omega=\omega\}, \\ \widetilde \Omega^k(N)  \simeq \Omega^{k}_{odd}(\widetilde N) = \{\omega \in \Omega^k(\widetilde{ N})\colon I^*\omega=-\omega\}.\end{gather}

These isomorphisms are provided by pull-backs via the projection $\widetilde N \to N$. Note that although the pull-back of a pseudo-form on $N$ is a pseudo-form on $\widetilde N$, i.e. a section of $ \bigwedge^k T^*\widetilde N \,\otimes \, o(\widetilde N)$, the canonical orientation on $\widetilde N$ gives us a non-vanishing section $\mathbf 1$ of $o(\widetilde N)$ (given by assigning $+1$ to every chart in a positively oriented atlas) and hence an identification  between forms and pseudo-forms given by $\omega \mapsto \omega \otimes \mathbf 1$.
Since the pull-back of a pseudo-form (as well as any other object) on $N$ to $\widetilde N$ is even, and the section $\mathbf 1$ is odd, we get an identification of pseudo-forms on $N$ with odd forms on $\widetilde N$.

\begin{example} 
\label{projective-pseudo-function}
Any homogeneous odd-degree polynomial $P \in \R[x, y, z]$  is an odd function on the unit sphere $S^2$ and hence defines a pseudo-function on the projective plane $\mathbb{RP}^2=S^2\,/\,\Z_2$. Likewise, any homogeneous even-degree polynomial defines a regular function on $\mathbb{RP}^2$.
\end{example}

Note that both even and odd $k$-forms can be integrated over compact $k$-dimensional submanifolds. To integrate an even form we need the submanifold to be oriented.  To integrate an odd form we need the submanifold to be \emph{co-oriented}. In particular, any density can always be integrated over the whole manifold, assuming the manifold is compact (as opposed to a volume form which can only be integrated over an oriented manifold).

\begin{example}
For any (orientable or not) compact Riemannian manifold $(N, g)$ the integral of the associated Riemannian density is well-defined (i.e. it does not depend on the orientation of $N$ or on whether such an orientation even exists). This integral is the Riemannian volume of $N$ and is a positive number.
\end{example}

{ Since, locally, one can identify odd and even forms, all standard local operations with even forms have odd counterparts.}
 In particular, the differential $\diff \alpha$ of an odd $k$-form $\alpha$ is an odd $(k+1)$-form, so that we have a map
$d\colon \widetilde \Omega^k(N) \to \widetilde \Omega^{k+1}(N).$ {Likewise, for any vector field $v$ on $N$ we have the interior product operator $i_v \colon \widetilde \Omega^{k}(N) \to \widetilde \Omega^{k-1}(N)$ and the Lie derivative $L_v = d i_v + i_v d \colon \widetilde \Omega^{k}(N) \to \widetilde \Omega^{k}(N)$. }
The exterior product of two forms of the same parity is an even form, while the product of two forms of different parity is an odd form. The latter, in particular, means that on any closed (i.e. compact without boundary) $n$-dimensional manifold $N$ one has a non-degenerate pairing $\widetilde \Omega^k(N) \times  \Omega^{n-k}(N) \to \R$. The annihilator of the space of exact (respectively, closed) $k$-forms of any given parity with respect to that pairing is the space of closed (respectively, exact) $(n-k)$-forms of opposite parity. This gives \emph{twisted Poincar\'e  duality}
$$
\widetilde H^k(N; \R) \simeq H^{n-k}(N; \R)^*,
$$
where $\widetilde H^{\bullet}(N; \R)$ is the cohomology of the cochain complex $(\widetilde \Omega^\bullet(N), d)$. One also has identifications
$$
H^k(N; \R) \simeq H^k_{even}(\widetilde N; \R), \quad \widetilde H^k( N; \R) \simeq H^k_{odd}(\widetilde N; \R),
$$
where $ H^k_{even}(\widetilde N; \R)$ and $H^k_{odd}(\widetilde N; \R)$ are, respectively $+1$ and $-1$ eigenspaces for the action of the orientation-reversing involution on the cohomology $H^k(\widetilde N; \R)$ of the orientation double cover. In those terms twisted Poincar\'e duality rewrites as 
$$
 H^k_{even}( \widetilde N; \R) \simeq H^{n-k}_{odd}(\widetilde N; \R)^*.
$$

\subsection{The group of measure-preserving diffeomorphisms, its Lie algebra, and the dual of the Lie algebra}
\label{section:diffeo}

Let $N$ be a closed (i.e. compact without boundary) non-orientable manifold with density (i.e. a non-vanishing top-degree pseudo-form) $\rho$. The group $\Diff_\rho(N)$ consists of measure-preserving diffeomorphisms: 
$$\Diff_\rho(N):=\{\Phi\in\Diffeo(N)\mid  \Phi^*\rho=\rho\}.$$
Let $\widetilde{ N}$ be the orientation double cover of $N$. Then $\Diff_\rho(N)$
can be identified with the subgroup of $\Diff_\mu(\widetilde{ N})$  consisting of volume-preserving diffeomorphisms of the orientation double cover that commute with the orientation-reversing involution (the volume form $\mu$ on $\widetilde N$ is constructed as a pull-back of the density $\rho$).

The Lie algebra $\Vect_\rho(N)$ of the group $\Diff_\rho(N)$ consists of divergence-free vector fields:
$$\Vect_\rho(N):=\{v \in\vect(N)\mid L_v \rho = 0\}.$$
With any (even) $1$-form $\alpha \in \Omega^1(N)$ we can associate a linear functional $\ell_\alpha$ on the Lie algebra $\Vect_\rho(N)$ given by 
$$
\ell_\alpha(v) := \int_N (i_v \alpha) \rho
$$
(recall that the integral of a density over the whole manifold is well-defined).
\begin{definition}
A linear functional $\ell\colon \Vect_\rho(N)\to \R$ is called \emph{regular} if there exists $\alpha \in \Omega^1(N)$ such that 
$\ell= \ell_\alpha$. 
Denote the space of regular functionals by $\Vect_\rho^*(N)$.
\end{definition}

\begin{proposition}
The space of regular functionals $\Vect_\rho^*(N)$ is isomorphic to the space of cosets $\Omega^1(N)\,/\,\diff \Omega^0(N)$ of $1$-forms modulo exact $1$-forms.
\end{proposition}


\begin{proof}
By definition of a regular functional, we have a surjective vector space homomorphism 
$\Omega^1(N) \to \Vect_\rho^*(N)$ given by
$
\alpha \mapsto \ell_\alpha.
$
We need to show that its kernel is the space of exact (even) $1$-forms. To that end, 
rewrite $\ell_\alpha(v)$ as
$$
\ell_\alpha(v) = \int_N \alpha \wedge i_v \rho.
$$
Since $v$ is an arbitrary divergence-free vector field, $i_v \rho$ is an arbitrary closed odd $(n-1)$-form. So, the kernel of the map $
\alpha \mapsto \ell_\alpha
$ is the annihilator of closed  odd $(n-1)$-forms under the pairing $\Omega^1(N) \times \widetilde \Omega^{n-1}(N) \to \R$, i.e. the space of exact even $1$-forms, as needed.

\end{proof}

The coadjoint action of the group $\Diff_\rho(N)$ on the regular dual  $\Vect_\rho(N)^* = \Omega^1(N)\,/\,\diff\Omega^0(N)$  coincides with the natural action of diffeomorphisms on (cosets of) $1$-forms:
$$\Ad_\Phi^*[\alpha]=[\Phi^*\alpha],$$
where $\Phi\in\Diff_\rho(N)$ is a measure-preserving diffeomorphism and $\alpha \in\Omega^1(N)$ is a 1-form. 
The focus of the present paper is on classifying all generic orbits of that action. Note that those orbits 
can also be interpreted as symplectic leaves of 
the  Lie-Poisson structure on $\Vect_\rho^*(N)$.

\section{Classification of generic pseudo-functions in 2D}
\label{section:func}

\subsection{Simple Morse pseudo-functions}
We begin with recalling standard definitions of a Morse function and a simple Morse function.
\begin{definition}
Let $N$ be a smooth manifold. A {smooth} function $F \colon N \to \R$ is called a \emph{Morse function}
if all its critical points are non-degenerate. A Morse function $F \colon N \to \R$ is \emph{simple} if all its critical values are distinct.
\end{definition}
Below we formulate analogous notions for pseudo-functions.

\begin{definition} 
Let $N$ be a non-orientable manifold. \begin{itemize} \item A pseudo-function 
$
F \in \widetilde \Omega^0(N)
$
is \emph{Morse} if its lift $ \widetilde F \in \Omega^0_{odd} (\widetilde N)$ to the orientation double cover $\widetilde{ N}$ is a Morse function. \item A pseudo-function 
$
F \in \widetilde \Omega^0(N)
$
is \emph{simple Morse} if its lift $ \widetilde F \in \Omega^0_{odd} (\widetilde N)$ is simple Morse. The latter in particular implies that $0$ is not a critical value of $\widetilde F$ (if there was a critical point where $\widetilde F = 0$, then its image under the orientation-reversing involution would give another critical point at the same zero level).
\end{itemize}

\end{definition}

The space of smooth pseudo-functions $\tilde\Omega^0(N)$
can be identified with the space~$\Omega^0_{odd}(\widetilde N)$ of odd smooth functions
$
\widetilde{ N}\to\R
$.

Consider the $C^2$-topology on the space~$
\widetilde \Omega^0(N) \simeq \Omega^0_{odd}(\widetilde N)
$ 
induced by $C^2$-topology on the space
$
\Omega^0(\widetilde{ N})
$ of all smooth functions on $\widetilde N$.

\begin{theorem}
\label{theorem_simple_pseudo}
Simple Morse pseudo-functions on $N$ form an open and dense subset in the space of smooth pseudo-functions in the $C^2$-topology.

\end{theorem}

The proof is based on the following version of Whitney's embedding theorem:

\begin{lemma}[On the realization of a fixed-point-free involution as the antipodal map]
\label{lemma_mfd_involution}
Let $M$ be a compact smooth $n$-dimensional manifold equipped with a fixed-point-free involution $I\colon M\to M.$ Then there exists an embedding $\Phi \colon M \to V$, where $V$ is a vector space of dimension $2n(2n+1)$, such that the following diagram commutes: 

\begin{equation}\label{comDiag}
\centering
\begin{tikzcd}
M \arrow{r}[]{\Phi}  \arrow{d}{I}&  V   \arrow{d}[]{-\id}\\
M \arrow{r}{\Phi}& V,
\end{tikzcd}
\end{equation}
where $-\id \colon V \to V$ is the antipodal map $x \mapsto -x$. In other words, any fixed-point-free involution of a compact manifold can be realized as restriction of the antipodal map.

\end{lemma}
\begin{remark}
Although the exact dimension of the vector space $V$ is irrelevant for our purposes, we note that    composing the embedding given by Lemma~\ref{lemma_mfd_involution} with a projection onto a suitable subspace, one can bring down the dimension of $V$ to $2n+1$ (as in weak Whitney's theorem).
\end{remark}

\begin{proof}[Proof of Lemma~\ref{lemma_mfd_involution}]
By Whitney's theorem there exists an embedding 
$\Phi_1 \colon M\to W$, where $W$ is a vector space of dimension $2n$. 

We extend it to an embedding 
$$\Phi_2 \colon M\to W
\oplus W,
\quad
\Phi_2(x) := (\Phi_1(x) , \Phi_1(I(x))).
$$
Then 
the following diagram commutes:
\begin{equation}\label{comDiag2}
\centering
\begin{tikzcd}
M \arrow{r}[]{\Phi_2}  \arrow{d}{I}&  W \oplus W    \arrow{d}[]{P}\\
M \arrow{r}{\Phi_2}& W \oplus W
\end{tikzcd}
\end{equation}
where the involution $P$ on $W \oplus W$ is  $(x,y)\mapsto (y,x).$ Furthermore, since the involution~$I$ is fixed-point-free, the image $\Phi_2(M)$ does not intersect the diagonal
$\Delta  \subset W\oplus W.$ Now, consider the map $$\Psi \colon W 
\oplus W \to W \oplus (W \otimes W), \quad
\Psi(x,y) := (x-y, (x-y) \otimes (x+y)).
$$
It is easy to see that $\Psi$ is an injective immersion away from the diagonal $\Delta$. Therefore, since $M$ is compact, the map $ \Psi \circ \Phi_2 \colon M \to W \oplus (W \otimes W)$ is an embedding.
Furthermore, we have $\Psi\circ P=-\Psi,$ so the following diagram commutes:
\begin{equation}\label{comDiag5}
\centering
\begin{tikzcd}
M \arrow{r}[]{\Phi_2}  \arrow{d}{I}&  W \oplus W    \arrow{d}[]{P} \arrow{r}[]{\Psi} & W \oplus (W \otimes W) \arrow{d}[]{-\id} \\
M \arrow{r}{\Phi_2}& W \oplus W \arrow{r}[]{\Psi} & W \oplus (W \otimes W).
\end{tikzcd}
\end{equation}
Taking $V:= W \oplus (W \otimes W)$ and $\Phi := \Psi \circ \Phi_2$ completes the proof.\qedhere

\end{proof}

\begin{proof}[Proof of Theorem~\ref{theorem_simple_pseudo}]
    It is clear that simple Morse pseudo-functions on $N$ form an open subset in the space of all smooth pseudo-functions, so we only need to show that it is dense. First we show that the set of all (not necessarily simple) Morse pseudo-functions is dense. That is equivalent to showing that any odd function $f \in \Omega^0_{odd}(\widetilde N)$ can be approximated by an odd Morse function. Thanks to  Lemma~\ref{lemma_mfd_involution}, we can assume that $\widetilde N$ is embedded in a vector space $V$, with the orientation-reversing involution given by the antipodal map $x \mapsto -x$.  
    Take an odd function $f  \in \Omega^0_{odd}(\widetilde N)$, and consider the family of functions
    $
    F_\ell := f + \ell
    $
    where $\ell \in V^*$.    Notice that all functions in this family are odd with respect to the antipodal map. 
    Furthermore, almost all functions in such a family are Morse \cite[Theorem 8.1.1]{shastri2011elements}, so one can indeed find an odd Morse function arbitrarily close to $f.$
    
    The second step is to show that simple Morse pseudo-functions on $N$ are dense in the space of all Morse pseudo-functions. To that end, notice that the corresponding argument for functions \cite[Proposition~1.2.12]{nicolaescu2007invitation} is \emph{local}: one can perturb a given Morse function in the neighborhood of a critical point to slightly change its critical value. But locally one can trivialize the orientation bundle $o(N)$ and thus identify pseudo-functions with functions. Therefore, the same argument works for pseudo-functions, proving the theorem.\qedhere
   
\end{proof}

\subsection{Reeb graphs with involution}

Below we classify simple Morse pseudo-functions on closed connected non-orientable 2D surfaces up to area-preserving diffeomorphisms. This is equivalent to symplectic classification of quadruples $(M,I, \omega, F)$ where \begin{itemize}
    \item $M$ is a closed connected surface with symplectic form $\omega$;
    \item $I\colon M\to M$ is a fixed-point-free anti-symplectic involution;
    \item $F\in \Omega^0_{odd}(M)$ is a simple Morse function anti-symmetric under the action of $I$.
\end{itemize}

Consider for a moment an arbitrary simple Morse function $F$ on a closed connected orientable surface $M$. Define an \emph{$F$-level} as a connected component of a level set of~$F.$ Non-critical $F$-levels are diffeomorphic to circles. The surface $M$ is a union of $F$-levels, which form a foliation with singularities. 
The base space of that foliation with the quotient topology is homeomorphic to a finite connected graph $\Gamma_F$ (see Figure~\ref{torusInt} for a Morse function $F$ on a torus) whose vertices correspond to critical values of $F.$ 
This graph $\Gamma_F$ is called the \emph{Reeb graph} 
of the function $F$ (also called the \emph{Kronrod graph}~\cite{Reeb, Kronrode}). By $\pi$ we denote the projection $M\to \Gamma_F.$
The function $F$ on $M$ descends to a function $f$ on the Reeb graph $\Gamma_F$. It is also convenient
to assume that $\Gamma_F$ is oriented: edges are oriented in the direction of increasing $f.$ 
One can recover the topology of a closed connected orientable surface from its Reeb graph by the following formula:
$
b_1(M) = 2b_1(\Gamma_F),
$
where $b_1(.) := \dim H_1(.)$ is the first Betti number. 

Now return to consideration of a simple Morse function $F$ which is anti-symmetric under the involution $I$. 

In that setting $I$ descends to an involution $\iota\colon \Gamma_F\to \Gamma_F$ such that $\iota^*f = -f$ (see Figure~\ref{torus}). Note that the involution $\iota$ is not necessarily fixed-point-free.

\begin{definition}\label{def:Reeb_graph}
    A \emph{Reeb graph with involution} $(\Gamma, \iota, f)$ is an oriented connected finite graph $\Gamma$ equipped with an involution 
    $
    \iota\colon \Gamma \to \Gamma
    $
    and a continuous function 
    $
    f:\Gamma \to \R,
    $ 
    with the following properties:
\begin{enumerate}[label=(\roman*)]
\item Each vertex of $\Gamma$ is either $1$-valent or $3$-valent;
\item The function $f$ is anti-invariant under the action of $\iota$ i.e. 
$ \iota^*f = -f;$
\item $f$ is strictly increasing along each edge of $\Gamma$;
\item $f$ takes distinct non-zero values at vertices of $\Gamma$.
\end{enumerate}
\end{definition}

\subsection{Measured Reeb graphs with involution}
In this section we discuss the notion of a \emph{measured} Reeb graph with involution. 
Recall that we consider quadruples $(M,I, \omega, F)$ where $(M, \omega)$ is a  closed connected symplectic surface, $I\colon M\to M$ is a fixed-point-free anti-symplectic involution, and  $F \in \Omega^0_{odd}(M)$ is a simple Morse function anti-symmetric under the action of $I$.

The natural projection map $\pi \colon M \to \Gamma_F$ induces a measure $\mu:=\pi_*\omega$ on the graph $\Gamma_F.$ 
According to \cite[Proposition 3.4]{izosimov2016coadjoint}, measure $\mu$ is \emph{log-smooth}  in the sense of \cite[Definition 3.5]{izosimov2016characterization} (in a nutshell, this means that the measure is smooth at interior points of all edges as well as at $1$-valent vertices, while at 3-valent vertices it has logarithmic singularities). Furthermore, this measure is invariant with respect to the involution 
$
\iota\colon \Gamma_F\to \Gamma_F.
$

\begin{definition}\label{def:measured_Reeb_graph}
    A \emph{measured Reeb graph with involution} $(\Gamma, \iota, f,\mu)$ is a Reeb graph $(\Gamma, \iota, f)$ with involution equipped with a log-smooth measure $\mu$ invariant under the involution $\iota$.
\end{definition}
\begin{definition}
    A map $\phi:\Gamma_1\to\Gamma_2$ between  two measured Reeb graphs with involution $(\Gamma_1, \iota_1, f_1,\mu_1)$ and $(\Gamma_2, \iota_2, f_2,\mu_2)$ is an  \emph{isomorphism} if it is an isomorphism of topological graphs which maps all objects in $\Gamma_1$ to the corresponding objects in $\Gamma_2$, i.e. $\phi \circ \iota_1 =
 \iota_2 \circ \phi$, $\phi^*f_2 = f_1$, and $\phi_*\mu_1 = \mu_2$.

\end{definition}
The following definition from \cite{izosimov2016coadjoint} makes sense regardless of the presence of an involution:

\begin{definition} \label{def:compatible}
 A measured Reeb graph $(\Gamma, \iota, f, \mu)$ is \emph{compatible} with a symplectic surface $(M, \omega)$ if 
$$
 2b_1(\Gamma) = b_1(M), \quad \int_\Gamma \diff \mu=\int_M\omega
$$ (where $b_1(.)$ stands for the first Betti number).

\end{definition}

For a non-orientable surface $N$ with density $\rho \in \widetilde \Omega^2(N)$, we say that 
 a measured Reeb graph $(\Gamma, \iota, f, \mu)$ is \emph{compatible} with $(N, \rho)$ if it is compatible with $(\widetilde N, I, \omega)$, where $\widetilde N$ is the orientation double cover of $N$, the symplectic form $\omega \in \Omega^2(\widetilde N)$ is the pull-back of $\rho$, and $I$ is involution on $\widetilde N$ such that $\widetilde N \, / \, I = N$. More explicitly, the compatibility conditions can be stated as follows:
 \begin{definition} \label{def:compatible2}
 A measured Reeb graph $(\Gamma, \iota, f, \mu)$ is \emph{compatible} with a non-orientable surface $N$ equipped with a {density} $\rho$ if 
 $$
b_1(\Gamma) = b_1(N), \quad 
\int_\Gamma \diff \mu=2\int_N\rho.$$
\end{definition}
\subsection{Proof of the classification theorem}

Here we prove the following classification result for {\it pseudo-functions on  non-orientable surfaces}:

\begin{theorem}\label{classification_of_pseudo-functions_symplectic}  
Let $N$ be a closed connected non-orientable surface equipped with a density (non-degenerate pseudo-form of top degree) $\rho$. Then there is a one-to-one correspondence between simple Morse pseudo-functions on $N$, considered up to area-preserving diffeomorphisms, and isomorphism classes of measured Reeb graphs with involution compatible with $(N,\rho)$.
\end{theorem}

Equivalently, and in more detail, this theorem can be formulated as follows.

\begin{theorem}\label{classification_of_functions_symplectic} 
Let $(M,\omega)$ be a closed connected  symplectic surface together with a fixed-point-free anti-symplectic involution $I\colon M\to M$. Then there is a one-to-one correspondence between odd (i.e. anti-symmetric under $I$) simple Morse functions on $M$, considered up to symplectic diffeomorphisms {which commute with $I$}, and isomorphism classes of measured Reeb graphs with involution compatible with $M$. In other words, the following statements hold.
\begin{longenum}
   \item Let $F,$ $G \in \Omega^0_{odd}(M)$ be two odd simple Morse functions. Then the following conditions are equivalent:
   \begin{longenum}
       \item There exists a symplectic diffeomorphism $\Phi \colon M \to M$ such that { $\Phi \circ I = I\circ\Phi$ and} $\Phi_*F=G$.
       \item Measured Reeb graphs with involution associated with $F$ and $G$ are isomorphic. 
   \end{longenum}
   Moreover, any isomorphism between measured Reeb graphs with involution associated with $F$ and $G$ 
   can be lifted to a symplectic diffeomorphism $\Phi \colon M \to M$ such that { $\Phi \circ I = I\circ\Phi$ and} $\Phi_*F=G$.
   \item For each measured Reeb graph with involution $(\Gamma, \iota, f, \mu) $ compatible  with $(M,\omega)$ there exists an odd simple Morse function $F \in \Omega^0_{odd}(M)$ such that the corresponding measured Reeb graph with involution is isomorphic to $(\Gamma, \iota, f, \mu) $.
\end{longenum}
\end{theorem}
\begin{proof}
We begin with part (i). The implication (a) $\Rightarrow$ (b) is immediate from the definition of a measured Reeb graph of a function. To prove (b) $\Rightarrow$ (a) we will show that any isomorphism between measured Reeb graphs with involution associated with $F$ and $G$ 
   can be lifted to a symplectic diffeomorphism $\Phi \colon M \to M$ such that { $\Phi \circ I = I\circ\Phi$ and} $\Phi_*F=G$.  Let $\phi \colon \Gamma_F \to \Gamma_G$ be an isomorphism of measured Reeb graphs with involution. Then, by \cite[Theorem 3.11(i)]{izosimov2016coadjoint}, it can be lifted to a symplectic diffeomorphism $\Phi \colon M \to M$ such that $\Phi_*F=G$. We need to show that $\Phi$ can be chosen so that it commutes with the involution $I$. The idea is to pick an arbitrary $\Phi$ and then compose it on the right with a suitable {``shear flow''}.  
 
The required shear flow on every level of $F$ will be simply a shift  along the Hamiltonian vector field $\omega^{-1}dF$  by an amount depending on the level.  To find the needed magnitude of the shift on every level, consider the commutator
   \begin{equation}\label{eq:comm}
          \Psi := [I, \Phi^{-1}] =  I \circ \Phi^{-1} \circ I \circ\Phi.
   \end{equation}

  This map $\Psi$ is symplectic and preserves each $F$-level. 
   Let $E$ be the union of open edges of $\Gamma_F$ containing points where $f = 0$. Let also $\pi$ be the projection $M \to \Gamma_F$. Then $M_0 := \pi^{-1}(E)$ is a union of open cylinders foliated into regular $F$-levels. For any smooth function $\psi$ on $E$ (where we define charts of $E$ by using the function $f$) denote by $S_\psi \colon M_0 \to M_0$ 
   a symplectomorphism defined as follows: every point $x \in M_0$ moves for time $\psi(\pi(x))$ along trajectories of the Hamiltonian vector field $\omega^{-1}dF$.

   \begin{lemma}\label{oddlemma}
       There exists a smooth function $\psi \colon E \to \R$ which is odd (i.e. $\iota^*\psi = -\psi$) and satisfies $\Psi\vert_{M_0} = S_\psi$, where  $\Psi\vert_{M_0}$ is the restriction of the commutator map~\eqref{eq:comm} to $M_0$.
   \end{lemma}
   \begin{proof}
       Since $\Psi$ is symplectic and preserves each $F$-level, there exists some (a priori not necessarily odd) smooth function $\psi \colon E \to \R$  such that $\Psi\vert_{M_0} = S_\psi$. Further, observe that since the Hamiltonian vector field $\omega^{-1}dF$ is even, we have \begin{equation}\label{sprop}
       S_\psi \circ I = I \circ S_{\iota^*\psi}.\end{equation} Furthermore, from the definition of the map $\Psi$ we get 
       $I \circ \Psi = \Psi^{-1}\circ I,$
       so, since $\Psi\vert_{M_0} = S_\psi$, we obtain
       $$
I \circ S_\psi = S_{-\psi} \circ I = I \circ S_{-\iota^* \psi},
       $$
       meaning that $S_{\psi + \iota^*\psi}$ is the identity map.     Now, let $t \colon E \to \R$ be a smooth function defined as follows: for every $x \in E$, it is equal to the period of the  Hamiltonian vector field $\omega^{-1}dF$ on the circle $\pi^{-1}(x)$. Then, since  $S_{\psi + \iota^*\psi} = \id$, we have that
       $$
        \lambda := \frac{\psi + \iota^*\psi}{t}
       $$
       is a continuous integer-valued function on $E$, i.e. an element of $H^0(E; \Z)$. Furthermore, since the period function $t$ is even, the same is true for the function $\lambda$, i.e.  $\lambda \in H^0_{even}(E; \Z)$. Consider the map $  H^0(E; \Z) \to H^0_{even}(E; \Z)$ given by $\eta \mapsto \eta + \iota^*\eta$. The image of that map consists of those integral cochains $\eta \in H^0_{even}(E; \Z)$ which take even values on all edges in $E$ fixed by $\iota$. We claim that $\lambda$ has that property and hence belongs to the image. Indeed, every edge $e$ fixed by $\iota$ has a point $x_e$ whose preimage under the projection $\pi \colon M \to \Gamma_F$ is an $F$-level $\gamma_e$ fixed by $I$ (that $x_e$ is the unique point on $e$ where $F=0$). Since $I$ has no fixed points, the restriction of $I$ to that level must be a half-period shift along the vector field $\omega^{-1}dF$. Likewise, the restriction of $I$ to the corresponding $G$-level $\Phi(\gamma_e)$ is half-period shift along $\omega^{-1}dG$. But since the diffeomorphism $\Phi$  maps the Hamiltonian vector field $\omega^{-1}dF$ 
   to the Hamiltonian field $\omega^{-1}dG$, it follows that the commutator $\Psi$ on the $F$-level $\gamma_e$ is the identity. Therefore, $\psi(x_e)$  must be an integer multiple of $T(x_e)$, which forces $\lambda(e)$ to be an even number. So indeed we have $\lambda =  \eta + \iota^* \eta$ for some $\eta \in H^0_{}(E; \Z)$, and replacing $\psi$ with $\psi - \eta t$ we get a function with desired properties.
   \end{proof}

   Back to the proof of the theorem, let $m < 0$ be a real number such that $\inf\nolimits_{x \in e} f(x) < m$ for any edge $e \in E$. Let also $\zeta \colon \R \to \R$ be an odd smooth function which is equal to $-1$ for $x < m$. Define a smooth function $\xi \colon E \to \R$ by
    $
    \xi :=  \frac{1}{2} ( 1 + \zeta \circ f   )\psi. 
  $
 Then $\xi$ is equal to $0$ near lower (i.e. where $f < 0$) endpoints of edges in $E$ and satisfies
    \begin{equation}\label{eq:th}
        \xi - \iota^* \xi = \psi
    \end{equation}
    everywhere in $E$. 
   Let $ \widetilde \Phi:= \Phi \circ S_{-\xi}.$ That is a symplectomorphism $M_0 \to M_0$ pushing $F$ to $G$. Moreover, 
 \begin{gather}
[I, \widetilde \Phi^{-1}] = I \circ \widetilde \Phi^{-1}  \circ  I \circ \widetilde \Phi  = I \circ S_{\xi} \circ \Phi^{-1} \circ I \circ \Phi \circ S_{-\xi} \\ =   S_{\iota^*\xi} \circ I \circ \Phi^{-1} \circ I \circ \Phi \circ  S_{-\xi} = S_{\iota^*\xi} \circ S_\psi \circ S_{-\xi} = \id,  
   \end{gather}
   where the {third} equality follows from \eqref{sprop} and the last one from \eqref{eq:th}. So, $ \widetilde \Phi$ commutes with $I$. Also, since $\xi = 0$ near lower endpoints of edges in $E$, one can smoothly extend $ \widetilde \Phi$ to
   $
   M_- := (M\setminus M_0) \cap \{F < 0\}
   $
   by setting $\widetilde \Phi := \Phi$ in $M_-$. Furthermore, since $ \Phi$ commutes with $I$ in $M_0$, it can be smoothly extended to  
    $
   M_+ := (M\setminus M_0) \cap \{F > 0\}
   $
   by setting $ \widetilde \Phi := I\widetilde \Phi I$ in $M_+$. That way we get an extension of $ \widetilde \Phi$ to all of $M$. It is symplectic, commutes with $I$, and maps $F$ to $G$, as needed. Thus, part (i) of the theorem is proved.
\par\medskip
  Now, let us prove part (ii). The proof consists of four steps: (1) we construct a function $F$ on a symplectic surface which has a given measured Reeb graph and is anti-invariant under an anti-symplectic map $I_1$ (which is not necessarily an involution); (2) we modify $I_1$ so that it becomes an involution, which we call $I_2$; (3) by composing $I_2$ with appropriate Dehn twists we turn it into a fixed-point-free involution $I_3$; (4) we find a symplectomorphism conjugating $I_3$ and $I$, which yields a function with desired properties.
 \par \smallskip
  \emph{Step 1.} By \cite[Theorem 3.11(ii)]{izosimov2016coadjoint}, there exists a (not necessarily odd) simple Morse function $F \colon M \to \R$ whose measured Reeb graph (without involution) is $(\Gamma, f, \mu)$. Consider also $-F$ as a simple Morse function on the symplectic surface $(M, -\omega)$. The measured Reeb graph of the latter function is $(\Gamma, -f, \mu)$. So, the involution $\iota$ is an isomorphism of measured Reeb graphs of $F$ and $-F$ and hence it lifts, by \cite[Theorem 3.11(i)]{izosimov2016coadjoint}, to a diffeomorphism $I_1 \colon M \to M$ such that $I_1^*F = -F$ and $I_1^*\omega = -\omega$. 
   \par \smallskip
\emph{Step 2.} Consider the restriction of $I_1^2$ to the set $M_0$ defined in the proof of part~(i). It is a symplectic diffeomorphism preserving $F$-levels and hence can be written as $S_j$ for a suitable smooth function $j \colon E \to \R$.
Furthermore, we claim that $j$ can be chosen to be even. The proof is similar to that of Lemma~\ref{oddlemma}: since the map $I_1^2$ commutes with $I_1$, which is a lift of $\iota$, we must have

       $$
        \lambda := \frac{j - \iota^*j}{t} \in H^0(E; \Z).
       $$
Furthermore, $\lambda$ is odd and hence can be written as $\eta - \iota^* \eta$ for some $\eta \in H^0(E; \Z)$ (in contrast to the map $  H^0(E; \Z) \to H^0_{even}(E; \Z)$ given by $\eta \mapsto \eta + \iota^*\eta$, the map $  H^0(E; \Z) \to H^0_{odd}(E; \Z)$ given by $\eta \mapsto \eta - \iota^*\eta$ is surjective). So, one can make the function $j$ even by replacing it with $j - \eta t$.

Consider now the function $\xi \colon E \to \R$ defined by $\xi := \frac{1}{2} ( 1 + \zeta \circ f   )j$, where $\zeta \colon \R \to \R$ is a function from the proof of part (i). Then
      $
        \xi + \iota^* \xi = \psi
    $, 
    and thus $I_2 := I_1 \circ S_{-\xi}$ is an involution $M_0 \to M_0$: 
    $$
I_2^2 =  I_1 \circ S_{-\xi} \circ  I_1 \circ S_{-\xi} =S_{-\iota^*\xi}\circ  I_1^2 \circ S_{-\xi} = S_{-\iota^*\xi}\circ  S_{\psi} \circ S_{-\xi} = \id.
    $$
     Furthermore, $I_2$ is anti-symplectic and takes $F$ to $-F$. Also observe that $I_2$ coincides with $I_1$ near preimages of lower endpoints of edges in $E$ and thus can be extended to $M_-$ by setting $I_2 := I_1$ in that domain. Similarly, in $M_+$ we set $I_2 := I_1^{-1}$. That way we  get an anti-symplectic involution $I_2 \colon M \to M$ which takes $F$ to $-F$. 
     
      \par \smallskip
     \emph{Step 3.} The set of fixed points of $I_2$ is a union of $F$-levels, each of which projects to a fixed point of $\iota$. So the fixed point set of $I_2$ is a union of finitely many circles. To show that such fixed points can be removed, it suffices to prove that we can get rid of one fixed circle. This is done by composing $I_2$ with a Dehn twist about that circle. Specifically, to get rid of a fixed circle $\pi^{-1}(x_0)$, where $x_0 \in \Gamma$ is a fixed point of $\iota$, and $\pi \colon M \to \Gamma_f$ is the projection, consider the edge $e$ of $\Gamma$ containing $x_0$. 
    Let $t \colon e \to \R$ be the period function defined as in Lemma \ref{oddlemma}: for every $x \in e$, it is equal to the period of the the Hamiltonian vector field $\omega^{-1}dF$ on the circle $\pi^{-1}(x)$. Since the involution $I_2$ preserves the Hamiltonian vector field $\omega^{-1}dF$, the period function $t$ is even: $\iota^*t = t$. Define a function $\eta \colon e \to \R$ by
      $
    \eta :=  \frac{1}{2} ( 1 + \zeta \circ f  )t,
 $  where $\zeta \colon \R \to \R$ is as above.
   Then     $
        \eta + \iota^* \eta = t,
   $ 
    and so the map $\widetilde I_2 := I_2 \circ S_\eta$, extended to the whole $M$ by setting $\widetilde I_2 := \id$ away from $\pi^{-1}(e)$, is again an involution:
    $$
\widetilde I_2^2 = I_2 \circ S_\eta \circ I_2 \circ S_\eta = S_{\iota^* \eta + \eta} = S_t = \id.
    $$
    Furthermore, just like $I_2$, the involution $\widetilde I_2$ is anti-symplectic and maps $F$ to $-F$. In addition to that, it has less fixed circles than $I_2$. Proceeding in this fashion, we finally get an anti-symplectic fixed-point-free involution $I_3$ such that $ I_3^*F = -F$. 
    
     \par \smallskip
     \emph{Step 4.} It follows from Moser's theorem for non-orientable surfaces \cite[p. 4894]{bruveris2018moser} that $ I_3 = \Phi I \Phi^{-1}$ for some symplectic diffeomorphism $\Phi$. But then $\Phi^*F$ is a function anti--invariant under $I$ whose measured Reeb graph with involution is the given one. Thus, Theorem \ref{classification_of_functions_symplectic} (and hence Theorem \ref{classification_of_pseudo-functions_symplectic})      is proved.
\end{proof}


\begin{remark}
The above description of invariants of pseudo-functions under the action of area-preserving diffeomorphisms extends to the case of non-orientable surfaces $N$ with boundary, cf. \cite{kirillov2023classification}. Recall that in the boundary case 
 \emph{simple Morse functions} $F$ have to satisfy the following conditions:
a) all critical points of $F$ are non-degenerate; 
b) $F$ does not have critical points on the boundary $\partial N$; 
c) the restriction of $F$ to the boundary $\partial N$ is a Morse function; and d) 
all critical values of $F$ and of its restriction $F|_{\partial N}$ are distinct.

A {\it pseudo-function} $F$ on a non-orientable surface $N$ with boundary
 is \emph{simple Morse} if its lift  $\widetilde F$ to the orientation double cover $\widetilde N$ is simple Morse. The Reeb graph of such pseudo-function $F$ (defined as the Reeb graph of $\widetilde F$) contains both solid edges  (corresponding to $\widetilde F$-levels diffeomorphic to a circle) and  dashed edges (corresponding to $\widetilde F$-levels diffeomorphic to a segment). That graph is equipped with an involution induced by the involution of $\widetilde N$. The involution on the graph cannot have fixed points on dashed edges, since the corresponding fixed-point-free involution on the surface $\widetilde N$ cannot map a segment to itself.
\end{remark}

\section{Classification of coadjoint orbits in 2D}
\label{section:orbits}
\subsection{Coadjoint orbits and pseudo-functions} \label{subsection:orbits_to_functions}
Let $\Diff_\rho(N)$ be the group of area-preserving diffeomorphisms of a non-orientable surface $N$ endowed with a density $\rho$. 
In this section we classify generic orbits of the coadjoint action of $\Diff_\rho(N)$ on its regular dual space $\Vect_\rho^*(N) = \Omega^1(N)\,/\,\diff\Omega^0(N)$. Recall that this action coincides with the natural action by pull-backs.

\par

Consider the mapping
    $$\curl \colon \Omega^1(N)\,/\,\diff \Omega^0(N) \to \widetilde \Omega^0(N),$$
defined by taking the \emph{vorticity} pseudo-function
$$
\curl[\alpha] := \frac{\diff \alpha}{\rho}.
$$   This mapping is well-defined on cosets since 
 $\diff(\alpha+\diff f)=\diff\alpha$. 
Furthermore, the mapping $\curl$ is a surjection onto the space $\widetilde \Omega^0(N)$ of pseudo-functions, since $H^2(N; \R) = 0$. Finally, observe that  the mapping $\curl$ is $\Diff_\rho(N)$-equivariant. In other words, the following diagram commutes for any $\Phi \in \Diff_\rho(N)$:

\begin{equation}\label{comDiag3}
\centering
\begin{tikzcd}
\Omega^1(N)\,/\,\diff \Omega^0(N) \ar{r}{\Phi^*} \ar{d}{\curl} & \Omega^1(N)\,/\,\diff \Omega^0(N) \ar{d}{\curl} \\
\widetilde \Omega^0(N) \ar{r}{\Phi^*} & \widetilde \Omega^0(N).
\end{tikzcd}
\end{equation}
\begin{definition}
A coset $[\alpha]\in \Omega^1(N)\,/\,\diff \Omega^0(N)$ is called \emph{simple Morse} if $\curl[\alpha]$ is a simple Morse function. A coadjoint orbit $\pazocal O$ is called \emph{simple Morse} if some (and hence every) coset $[\alpha]\in \pazocal O$ is simple Morse. 
\end{definition}
With every simple Morse coset $[\alpha]\in \Omega^1(N)\,/\,\diff \Omega^0(N)$ one can associate a measured Reeb graph $\Gamma_{\curl[\alpha]}$ with involution. If two simple Morse cosets $[\alpha]$ and $[\beta]$ belong to the same coadjoint orbit then the corresponding Reeb graphs are isomorphic.

The converse statement is more subtle. Indeed, suppose that cosets $[\alpha]$ and $[\beta]$ have isomorphic Reeb graphs. Then it follows from Theorem~\ref{classification_of_pseudo-functions_symplectic} that there exists an area-preserving diffeomorphism $\Phi$ such that $\Phi^*\curl[\beta]=\curl[\alpha].$ Therefore, the 1-form $\Phi^*\beta-\alpha$ is closed. Since this 1-form is not necessarily exact, the cosets $[\alpha]$ and $[\beta]$ do not necessarily belong to the same coadjoint orbit. Nevertheless, we conclude that the space of coadjoint orbits 
corresponding to the same measured Reeb graph with involution is finite-dimensional and its dimension is at most $\dim \Hom^1(N; \R).$ 

\subsection{Even circulation functions on Reeb graphs with involution}
In order to obtain a complete set of invariants of simple coadjoint orbits for the group $\Diff_\rho(N)$, we lift all objects discussed in the previous section to the orientation double cover $M=\widetilde N$ of $N$. That orientation double cover is a symplectic surface with a symplectic form $\omega$ and a fixed-point-free anti-symplectic involution $I$. Our aim is to classify simple Morse cosets $\Omega^1_{even}(M) \,/\, \diff \Omega^0_{even}(M)$ up to even (i.e. commuting with $I$) symplectic diffeomorphisms.  To that end we employ the notion of a circulation function introduced in \cite{izosimov2016coadjoint}. 

Consider a simple Morse coset $[\alpha]\in \Omega^1_{even}(M) \,/\, \diff \Omega^0_{even}(M)$. Then $F := \diff \alpha / \omega$ is an odd simple Morse function on $M$. Let $\Gamma$ be the set of  $F$-levels. Recall that this set has a structure of a measured Reeb graph with involution $\iota$, and such graphs classify pseudo-functions up to area-preserving diffeomorphisms. To obtain classification of orbits, we define an additional structure on $\Gamma$.

Let $\pi \colon M \to \Gamma$ 
be the natural projection. Take any point $x$ lying in the interior of some edge 
$e $ of $ \Gamma$. Then $\pi^{-1}(x)$ is a closed curve $C$ in $M$. It is naturally oriented by the Hamiltonian vector field $\omega^{-1}dF$. The integral of $ \alpha$ 
over $C$ does not change if we change $ \alpha$ by a function differential.
Thus, we obtain a function
$
\circulation \colon \Gamma \setminus V( \Gamma) \to \R
$
given by 
$$
\circulation(x) = \oint\limits_{{\pi^{-1}(x)}}\oneform\,
$$
and defined outside of the set of vertices    $V( \Gamma)$ of the graph $ \Gamma$.

\begin{proposition}\label{circProperties}
The function $\circulation$ has the following properties.
\begin{longenum}
\item Assume that $x,y$ are two interior points of an edge $e$ of $ \Gamma$. Then
\begin{equation}\label{stokes}
    \circulation(y) - \circulation(x) = \int\limits_x^y f\diff \mu.
\end{equation}

\item Let $v$ be a vertex of $\Gamma$. 
Then $ \circulation$ satisfies the Kirchhoff rule at $v$:
\begin{align}
\label{3valentcirc}
 \sum_{{e \to v}} \lim\nolimits_{{x \xrightarrow[]{e} v }} \circulation(x)= \sum_{{e \leftarrow v}} \lim\nolimits_{{x \xrightarrow[]{e} v }} \circulation(x)\,,
\end{align}
where $ \sum_{{e \to v}}$ stands for summation over edges pointing at the vertex $v$,  $\sum_{{e \leftarrow v}}$ stands for summation over edges pointing away from $v$, and $x \xrightarrow[]{e} v$ means $x \in \Gamma \setminus V(\Gamma)$ tends to $v$ along $e$.
\item The function $\circulation$ is even with respect to the involution $\iota$ on $\Gamma$.
\end{longenum}
\end{proposition}
\begin{proof}
The first two properties hold regardless of the presence of involution  \cite{izosimov2016coadjoint}. The last property holds because the form $\alpha$ and the vector field $\omega^{-1}dF$ are both even.
\end{proof}

\begin{definition}
{\rm
Let  $(\Gamma, \iota, f, \mu)$ be a measured Reeb graph with involution. Any function $\circulation \colon \Gamma \setminus V(\Gamma) \to \R$ satisfying properties listed in Proposition \ref{circProperties} is called an \textit{even circulation function}. A measured Reeb graph with involution endowed with an even circulation function is called a \textit{circulation graph with involution}  $(\Gamma, \iota, f, \mu, \circulation)$.

Two circulation graphs with involution are isomorphic if they are isomorphic as measured Reeb graphs with involution, and the isomorphism between them preserves the circulation function.

}
\end{definition}
Above we associated a circulation graph with involution   $\Gamma_{[\oneform]} :=(\Gamma, \iota, f, \mu, \circulation)$ to 
any simple Morse coset $[\alpha]\in \Omega^1_{even}(M) \,/\, \diff \Omega^0_{even}(M)$.
\begin{remark}

    Note that the function $f$ on a circulation graph can be recovered from the circulation function 
$\circulation$, as \eqref{stokes} implies $f= \diff \circulation/\diff \mu$.
\end{remark}

The following result describes the space of even circulation functions on a given measured Reeb graph with involution.

\begin{proposition}
The space of even circulation functions on a measured Reeb graph with involution $(\Gamma, \iota, f, \mu)$ is an affine space whose associated vector space is $\Hom^{odd}_1(\Gamma; \R) := \{ \lambda \in \Hom_1(\Gamma; \R) \mid \iota_*\lambda = -\lambda \}$.
\end{proposition}
\begin{proof}
By definition, a function $\circulation \colon \Gamma \setminus V(\Gamma) \to \R$ is an even circulation function if it satisfies certain inhomogeneous linear equations. So, the set of even circulation functions on $\Gamma$ is indeed an affine space. Let us first show that it is non-empty. To that end, observe that since $f$ is odd, we have $\int_\Gamma f \diff  \mu = 0$, so by \cite[Propoisition 4.5(i)]{izosimov2016coadjoint} the measured Reeb graph $(\Gamma, f, \mu)$ admits a circulation function $\circulation$. Furthermore, the latter can be made even by considering the averaged function $\frac{1}{2}(\circulation + \iota^*\circulation)$. So, the space of even circulation functions is a solution space of a consistent inhomogeneous linear system, which means that the corresponding vector space is the solution space of the associated homogeneous system. That solution space consists of even functions  $\xi \colon \Gamma \setminus V(\Gamma) \to \R$ which are constant on each edge and satisfy Kirchhoff's rule at each vertex. For each element $\xi$ of that solution space, consider a $1$-chain on $\Gamma$ given by
    $
      \lambda(\xi) := \sum \xi\vert_e \cdot e,
    $
    where the sum is over all edges of $\Gamma$. Then Kirchhoff's equations on $\xi$ are equivalent to $\lambda(\xi)$ being a cycle, i.e. $\lambda(\xi) \in \Hom_1(\Gamma; \R)$. Furthermore, since the involution $\iota$ reverses orientation of edges, $\xi$ is even if and only if $\lambda(\xi)$ is odd. So, the vector space associated with the affine space of even circulation functions on $\Gamma$ is indeed $\Hom^{odd}_1(\Gamma; \R)$, as claimed. 
\end{proof}

\begin{corollary}
    The dimension $d$ of the space of even circulation functions on a measured Reeb graph with involution $(\Gamma, \iota, f, \mu)$ is given by
     \begin{equation}\label{eq:cirdim}
   d = \dim \Hom^{odd}_1(\Gamma; \R) = \frac{1}{2}(  \#\mathrm{Fix}(\iota) +  b_1(\Gamma) - 1),
    \end{equation}
    where $\#\mathrm{Fix}(\iota)$ is the number of fixed points of $\iota$, and $b_1(\Gamma) = \dim H_1(\Gamma, \R)$ is the first Betti number of $\Gamma$. In particular,
   \begin{equation}\label{eq:circin}
       \frac{1}{2}(   b_1(\Gamma) - 1)  \leq d \leq b_1(\Gamma).
    \end{equation}
\end{corollary}
\begin{proof}
By the Hopf trace formula we have
\begin{gather}
\#\mathrm{Fix}(\iota) = 1 -  \dim \Hom^{even}_1(\Gamma; \R) +  \dim \Hom^{odd}_1(\Gamma; \R) \\ = 1 -  \dim \Hom^{}_1(\Gamma; \R)  +  2\dim \Hom^{odd}_1(\Gamma; \R),
\end{gather}
hence the result.
\end{proof}

\begin{remark}
Another way to express this dimension is $d= b_1(\Gamma)- b_1(\Gamma / \iota)$. Indeed, this follows from the fact that  odd classes form the kernel of the projection $\Hom^{}_1(\Gamma; \R) \to \Hom^{}_1(\Gamma / \iota, \R)$. 
\end{remark}

\begin{remark}
The inequality $d \leq b_1(\Gamma)$ holds since the space $\Hom^{odd}_1(\Gamma; \R)$ is a subspace of $\Hom^{}_1(\Gamma; \R)$. That inequality is also equivalent to $\#\mathrm{Fix}(\iota) \leq b_1(\Gamma) + 1$. The latter is true since the set $\{f = 0\}$ splits $\Gamma$ into two connected components and hence consists of at most $b_1(\Gamma) + 1$ points, while the fixed point set $\mathrm{Fix}(\iota)$ is a subset of $\{f = 0\}$.

It is easy to see that there are no other restrictions on the number $\#\mathrm{Fix}(\iota)$ in addition to $0 \leq \#\mathrm{Fix}(\iota) \leq b_1(\Gamma) + 1$ and $\#\mathrm{Fix}(\iota) \equiv b_1(\Gamma) + 1 \mod 2$, so that all integer dimensions $d$ satisfying \eqref{eq:circin} can occur.
\end{remark}

\begin{example}\label{ex:rp2graph}
Assume that the graph $\Gamma$ is a tree. Then $\dim \Hom^{odd}_1(\Gamma; \R) = 0$, so there is a unique even circulation function on $\Gamma$.

\end{example}

\begin{example}\label{ex:kleingraph}
Assume that $b_1(\Gamma) = 1$. Then inequalities \eqref{eq:circin} imply that the dimension $d$ of the space of even circulation functions on $\Gamma$ is $0$ or $1$. Furthermore, by  formula \eqref{eq:cirdim} we have that $d = 0$ if and only if the involution $\iota$ on $\Gamma$ has no fixed points. An example of such an involution is shown in Figure \ref{torus} in the introduction. As for the case~$d=1$, that corresponds to two fixed points, see Figure \ref{torusgraph2}.

\end{example}

\subsection{Proof of the classification theorem}
The main result of this section is the following classification of generic coadjoint orbits for
the group of measure-preserving diffeomorphisms of a non-orientable surface:

\begin{theorem}\label{thm:sdiffN}
Let $N$ be a closed connected non-orientable surface equipped with a density $\rho.$ Then simple Morse coadjoint orbits of $\Diff_\rho(N)$
are in one-to-one correspondence with isomorphism classes of circulation graphs compatible with $(N,\rho)$. 
\end{theorem}

Compatibility of a graph and non-orientable surface is understood as in Definition~\ref{def:compatible2}. Since a circulation graph is also a measured Reeb graph (with an additional structure), the definition applies.

An equivalent and more detailed form of this classification, which we are going to prove, can be formulated in terms of the corresponding orientation double cover:

\begin{theorem}\label{thm4} \label{thm:sdiffM} 
Let $(M, I, \omega)$ be a closed connected symplectic surface equipped with a fixed-point-free anti-symplectic involution $I$. Then generic orbits of the action of even (i.e. commuting with $I$) symplectomorphisms of $M$ on the coset space $[\alpha]\in \Omega^1_{even}(M) \,/\, \diff \Omega^0_{even}(M)$ are in one-to-one correspondence with (isomorphism classes of) circulation graphs compatible with $M$ (in the sense of Definition \ref{def:compatible}).
In other words, the following statements hold:
\begin{longenum}
\item For two simple Morse cosets $[\oneform], [\oneformtwo] \in  \Omega^1_{even}(M) \,/\, \diff \Omega^0_{even}(M)$, the following conditions are equivalent:
\begin{longenum} \item $\Phi_*[\oneform] =  [\oneformtwo]$ for some even symplectomorphism $\Phi$;  \item circulation graphs $\Gamma_{[\oneform]}$ and $\Gamma_{[\oneformtwo]}$ corresponding to the cosets $[\oneform]$ and $[\oneformtwo]$ 
are isomorphic.\end{longenum}
\item For each circulation graph $(\Gamma, \iota, f, \mu, \circulation)$ which is compatible 
 with $(M,\omega)$, there exists a simple Morse coset $  [\alpha]\in \Omega^1_{even}(M) \,/\, \diff \Omega^0_{even}(M)$ such that  $\Gamma_{[\oneform]} =(\Gamma, \iota, f, \mu, \circulation)$.
\end{longenum}
\end{theorem}

\begin{proof}
    We first prove part (i). The implication (a) $\Rightarrow$ (b) is by construction, so we only need to prove  (b) $\Rightarrow$ (a). In view of Theorem \ref{classification_of_functions_symplectic}, it suffices to consider the case $\diff  [\alpha] = \diff  [\beta] = F\omega$ and prove that if the circulation functions on the graph~$\Gamma$ of~$F$ given by cosets $[\oneform], [\oneformtwo]$ are the same, then there is an even symplectic diffeomorphism $\Phi \in \Diff_\omega(M)$ such that $\Phi^* [\beta] = [\alpha]$. Consider $\xi := [\alpha] - [\beta]$. Then~$\xi \in \Hom^1_{even}(M)$. Furthermore, since the circulation functions of $[\oneform]$ and $ [\oneformtwo]$ coincide, it follows that the class $\xi$ has zero periods over $F$-levels. Therefore, by \cite[Lemma 4.8]{izosimov2016coadjoint}, there exists a smooth function $G \in \Omega^0(M)$   
    such that the $1$-form $GFdF$ is closed and its cohomology class is $\xi$. (The lemma says that there is $H \in \Omega^0(M)$ such that $HdF$ is closed and its class is $\xi$. Furthermore, that $H$ is divisible by $F$ in $\Omega^0(M)$, so we just set $G := H/F$.) Furthermore, since the class $\xi$ is even, without loss of generality we can assume that $G$ is even as well, i.e.  $G \in \Omega^0_{even}(M)$ (if not, we replace $G$ with $\frac{1}{2}(G + I^*G)$).

    Consider even symplectic vector field $X := G \omega^{-1}dF$. Then, for the flow $\Phi_t$ of $X$, we have
    $$
    \frac{d}{dt} \Phi_t^* [\beta] =  \Phi_t^* L_X [\beta] = \Phi_t^* [i_X F\omega] = [GFdF] = \xi.
    $$
    In particular, the time-$1$ flow $\Phi_1$ of $V$ takes $[\beta]$ to $[\alpha]$, as needed.

    \smallskip

    We now prove part (ii). By Theorem \ref{classification_of_functions_symplectic}, there exists an odd simple Morse function $F \in \Omega^0_{odd}(M)$ whose measured Reeb graph with involution is $(\Gamma, \iota, f, \mu)$. We need to show that the map from the affine space of cosets
    $
     [\alpha]\in \Omega^1_{even}(M) \,/\, \diff \Omega^0_{even}(M) $ such that $\diff \alpha = F \omega$ to the space of even circulation functions on $\Gamma$, given by mapping a coset $[\alpha]$ to the associated circulation function $\circulation_{[\alpha]}$, is surjective. To that end consider the associated map of vector spaces $ \Hom^1_{even}(M ; \R) \to \Hom_1^{odd}(\Gamma; \R)$ which takes a class $[\beta] \in \Hom^1_{even}(M ; \R)$ to a chain $\sum \beta(e)e$ where $\beta(e)$ is the integral of $\beta$ over the preimage of any interior point of $e$ under the projection $\pi \colon M \to \Gamma$. Upon identification $\Hom^1_{even}(M ; \R) \simeq \Hom_1^{odd}(M ; \R)  $ given by (twisted) Poincar\'e duality, that vector space map becomes the projection $\pi_* \colon \Hom_1^{odd}(M ; \R) \to  \Hom_1^{odd}(\Gamma; \R) $, which is surjective. Therefore, the map $[\alpha] \mapsto \circulation_{[\alpha]}$ between affine spaces is surjective as well.  Thus, the theorem is proved.
\end{proof}

\begin{corollary}\label{cor:dim}
Let $N$ be a closed connected non-orientable surface equipped with a density $\rho$.
Then the space of coadjoint orbits of the group $\Diff_\rho(N)$ corresponding to the same measured Reeb graph $(\Gamma, \iota, f,\mu)$ is an affine space of dimension 
    $$
   d = \dim \Hom^{odd}_1(\Gamma; \R) = \frac{1}{2}(  \#\mathrm{Fix}(\iota) +  b_1(N) - 1),
    $$
    where $\#\mathrm{Fix}(\iota)$ is the number of fixed points of $\iota$, and $b_1(N) = \dim H_1(N, \R)$ is the first Betti number of $N$. In particular,
   \begin{equation}
      \frac{1}{2}(  b_1(N) - 1)  \leq d \leq b_1(N).
    \end{equation}
\end{corollary} 
Note that for an orientable surface $M$ the corresponding dimension $d$ is always $ \frac{1}{2}   b_1(M)$, i.e. the genus of $M$.

\begin{example}\label{ex:RP2}
Consider the projective plane $\RP^2.$ The first homology group $\Hom_1(\RP^2; \R)$ is trivial. Therefore, in this case there is a one-to one correspondence between generic coadjoint orbits and measured Reeb graphs with involution, in agreement with Example \ref{ex:rp2graph}. 
\end{example}

\begin{example}\label{ex:Klein3p} 
Here we elaborate on Example \ref{ex:Klein3} from the introduction. The function on a torus shown in Figure~\ref{torus} defines  a pseudo-function on the Klein bottle $K^2.$ One has $b_1(K^2)=1$, while the involution $\iota$ has no fixed points. Therefore, in this case the space of coadjoint obits corresponding to the given measured Reeb graphs with involution is $0$-dimensional, i.e. the graph completely determines the orbit, just like in Example \ref{ex:RP2}.

Now consider a donut lying on a horizontal table, and let $F$ be the height function on its surface, normalized so that the center of symmetry of the donut is at height $0$. Then $F$ is odd with respect to the central symmetry.  Furthermore, even though $F$ is not a Morse function (its critical points are degenerate and form two circles), we can still consider the corresponding graph $\Gamma_F$ defined as the set of $F$-levels with quotient topology, and that graph is equipped with an involution $\iota$ induced by central symmetry of the donut. Topologically, the graph $\Gamma_F$ is a circle, while the involution $\iota$ is given by axial symmetry and has two fixed points. Now consider a small odd Morse perturbation of $F$ (e.g. consider the height function for a donut on a slightly inclined table). Then each critical circle of $F$ will fall apart into two Morse critical points, and the resulting graph with involution will be as shown in Figure \ref{torusgraph2}: by continuity the involution on the graph still has two fixed points. The so-obtained function on the torus can again be thought as a pseudo-function on the Klein bottle $K^2$. By Corollary \ref{cor:dim}, the space of coadjoint orbits of $\Diff_\rho(K^2)$ corresponding to such a function is $1$-dimensional, as opposed to the height function on a ``standing torus'' where the dimension of the orbit space is $0$. Note that $0$ and $1$ are the only possible dimensions of the orbit space for the Klein bottle, see Example \ref{ex:kleingraph}.
\end{example}

\medskip

\medskip

\begin{appendices}

\section{Motivation: The Hamiltonian framework of the Euler equation}\label{sect:hydro}

\subsection{The Euler equation}

The main motivation for classification of {coadjoint orbits for the group of measure-preserving diffeomorphisms} is related to description of first integrals for the Euler equation of ideal hydrodynamics. 
Consider an inviscid incompressible  fluid filling a compact{, possibly non-orientable,} $n$-dimensional Riemannian manifold  $M$ 
with the Riemannian density form $\rho$. 
The motion of an inviscid incompressible  fluid filling 
$M$  is governed by the hydrodynamic Euler equation
\begin{equation}\label{idealEuler}
\partial_t u+\nabla_u u=-\nabla p
\end{equation}
on  the divergence-free velocity field $u$ of a fluid flow in $M$. Here $\nabla_u u$ stands for  the Riemannian 
covariant derivative of the field $u$ along itself, while the function $p$  is determined by the divergence-free 
condition up to an additive constant.

Arnold  \cite{arnold1966geometry} showed that the Euler equation can be regarded as an 
equation of the geodesic flow on the group $\Diff_\rho(M):=\{\phi\in \Diff(M)~|~\phi_*\rho=\rho\}$ 
of measure-preserving diffeomorphisms of $M$ with respect to a right-invariant metric on the group 
given at the identity by the $L^2$-norm of the fluid's velocity field.
This geodesic description implies  the following Hamiltonian framework for the Euler equation.  
Consider the regular dual space $\mathfrak g^*=\Vect_\rho^*(M)$ 
of the Lie algebra $\mathfrak g={\Vect}(M)=\{u\in {\Vect}(M) \mid L_u\rho=0\}$ of divergence-free vector fields on $M$.
This dual space  has a natural  description 
as the  space of cosets  $\Vect_\rho^*(M)=\Omega^1(M) \,/\, \diff \Omega^0(M)$, where $\Omega^k(M)$ is the space of smooth $k$-forms on $M$.
For a 1-form $\alpha$ on $M$ its coset of 1-forms is 
$$
[\alpha]=\{\alpha+df \mid f\in C^\infty(M)\}\in \Omega^1(M) \,/\, \diff \Omega^0(M)\,.
$$
The pairing between cosets and divergence-free vector fields is given by 
$$ [\alpha](u):=\int_M\alpha(u)\,\rho$$ for any  field $u\in {\Vect_\rho}(M)$ 
(note that for a non-orientable manifold $M$ this integral is still well-defined as the integral of a function $\alpha(u)$ against the density pseudo-form $\rho$). This pairing is well-defined on cosets because the latter integral vanishes for any exact $1$-form $\alpha$ and any $u \in \Vect_\rho(M)$.
The coadjoint action of the group $\Diff_\rho(M)$ on the dual 
 $\mathfrak g^*$ is given by the change of coordinates in (cosets of) 1-forms on $M$ 
 by means of measure-preserving diffeomorphisms.

The Riemannian metric $(\,,)$ on the manifold $M$ allows one to
identify the Lie algebra and its regular dual by means of the so-called inertia operator:
given a vector field $u$ on $M$  one defines the 1-form $\alpha=u^\flat$ 
as the pointwise inner product with  the velocity field $u$:
$u^\flat(v): = (u,v)$ for all $v\in TM$. 
The Euler equation rewritten on 1-forms $\alpha=u^\flat$ becomes
$\partial_t \alpha+L_u \alpha=-dP$
for  an appropriate function $P$ on $M$.
In terms of the cosets of 1-forms $[\alpha]$, the Euler equation on the dual space $\mathfrak g^*$ takes the form
\begin{equation}\label{1-forms}
\partial_t [\alpha]+L_u [\alpha]=0\,.
\end{equation}
The latter Euler equation on $\mathfrak g^*=\Vect_\rho^*(M)$ turns out to be a Hamiltonian equation with the Hamiltonian functional   given by the fluid's kinetic energy.
The  Hamiltonian operator is given by the Lie algebra coadjoint action ${\rm ad}^*_u$, 
which  in the case of the diffeomorphism group corresponds to the Lie derivative: ${\rm ad}^*_u=L_u$.
The symplectic leaves of the corresponding Lie-Poisson structure on the dual space $\mathfrak g^*$ 
are coadjoint orbits of the corresponding group $\Diff_\rho(M)$, see details in \cite{arnold1966geometry, arnold1999topological}.
All invariants of the coadjoint action, also called Casimirs,  are first integrals of the Euler equation
for {\it any choice} of Riemannian metric.

\subsection{Vorticity and Casimirs}
 Introduce the {\it vorticity 2-form} $\xi:=\diff u^\flat$ as the differential of the 1-form 
$\alpha=u^\flat$ on $M$ and note  that  the vorticity exact 2-form is well-defined for cosets $[\alpha]$: 
1-forms $\alpha$ in the same coset have equal vorticities $\xi=\diff \alpha$. 
Written in terms of $\xi$, the Euler equation~\eqref{1-forms} assumes the vorticity (or Helmholtz) form $$\partial_t \xi+L_u \xi=0,$$
which means that the vorticity form is transported by (or ``frozen into'') the fluid flow (Kelvin's theorem).
The definition of vorticity $\xi$ as an exact 2-form $\xi=\diff u^\flat$ makes sense for a manifold $M$ 
of any dimension and regardless of orientability of~$M$.

In this paper{,} we consider the case of a non-orientable manifold of dimension~$2$. In this setting, the vorticity of the fluid can be regarded as a {\it vorticity pseudo-function} $F=\diff u^\flat/\rho$, where $\rho$ is the Riemannian area form. Since the vorticity pseudo-function  $F$ is transported by the flow, one can define Casimirs generalizing the notion of enstrophies for orientable surfaces, which are all moments
of  $F$ over the surface $M$. In the non-orientable case, one considers only \emph{even} moments

$$
 I_{k}(F):=\int_{M} \!\! F^k\,\rho, ~k=0,2,4,...
$$
which are well-defined since even powers of a pseudo-function are functions and can be integrated against a density. 

To obtain a full set of Casimirs (invariants of  the $\Diff_\rho(M)$-action), one needs to consider similar integrals but computed for each edge of the Reeb graph separately, cf. \cite{izosimov2017classification}. {Furthermore, if $H^1(M, \rho) \neq 0$, one needs to supplement the so-defined moments with invariants which determine the circulation function. In hydrodynamical terms, those invariants can be viewed as fluid circulation, i.e., integrals of the $1$-form $\alpha = u^\flat$, along appropriately defined cycles on $M$.}

\section{Singular coadjoint orbits in the non-orientable setting}\label{section:vort}

Assume that $M$ is a closed, possibly non-orientable,  manifold with $H^1(M, \R) = 0$. Then the regular dual $\Vect_\rho^*(M)$ to the algebra of divergence-free vector fields on $M$ can be dentified with the space $d\Omega^1(M)$ of exact $2$-form on $M$. The coadjoint action of $\Diff_\rho(M)$  on this space is given by pull-back of $2$-forms. Here, we consider an extension of this action to singular exact $2$-forms (de Rham currents), in particular those of the form $\delta_P$, where $P \subset M$ is a null-homologous codimension $2$ closed submanifold. In the context of hydrodynamics, voticities of the form $\delta_P$ correspond to point vortices in dimension $2$, and \emph{vortex membranes} in higher dimensions.

\subsection{Marsden--Weinstein symplectic structure on co-oriented submanifolds}

Let $N$ be a closed non-orientable manifold. In view of the pairing $$ \Omega^k(N) \times \widetilde \Omega^{n-k}(N) \to \R,$$ it is natural to define a \emph{singular $k$-form} (equivalently, a de Rham current of degree $n-k$) as a continuous (in appropriate sense) functional on $ \widetilde \Omega^{n-k}(N)$. In particular, for a codimension $2$ co-oriented closed submanifold $P \subset N$, one has a singular $2$-form $\delta_P$ whose value on a pseudo-form $\omega$ of degree $n-2$ is given by
$\int_P \omega$.

\begin{proposition}
The current $\delta_P$ is closed. Furthermore, if  $P$ is a boundary, then $\delta_P$ is exact.
\end{proposition}
\begin{proof}
Closed currents are those which vanish on exact pseudo-forms. This is clearly the case for $\delta_P$, by Stokes theorem for non-orientable manifolds. Furthermore, if $P = \partial Q$, then $\int_P \omega = \int_Q d\omega = 0$ for any closed pseudo-form $\omega$ of degree $n-2$, so $\delta_P$ is exact.
\end{proof}

\medskip

We will call codimension $2$ co-oriented closed null-homologous  submanifolds $P$ in $ N$ vortex membranes.  The tangent space  at $P$ to the space of all vortex membranes in $N$ can be naturally identified with the space of smooth sections of the normal bundle to $P$.
Let $V$ and $W$ be a pair of such sections. Suppose that $N$ is equipped with a density $\rho$. 

\begin{definition}
    The {\it generalized  Marsden--Weinstein symplectic structure} on the space of membranes is defined as
$$
\omega_P^{\text{MW}}(V, W):=\int_P i_V i_W \rho.
$$
\end{definition}

Since $P$ is co-oriented and $\rho$ is a pseudo-form of degree $n = \dim N$, the inner product  $i_V i_W \rho$ is a well-defined pseudo-form of degree $n-2$ on $P$. Therefore, it can be integrated over $P$, and the corresponding  integral  is well defined.
One can see that it is skew-symmetric and non-degenerate on normal vector fields $V$ and $W$.
As we will see below, it is also closed and hence defines a symplectic structure on the space of membranes.

Recall that each vortex membrane $P \subset N$ gives rise to a singular exact $2$-form $\delta_P$. As a result, one can view the space of vortex membranes $P$ (equivalently, the space of the corresponding forms $\delta_P$) as a singular coadjoint orbit of $\Diff_\rho(N)$. The action of $\Diff_\rho(N)$ on vortex membranes is just the natural action of diffeomorphisms on submanifolds. Thus, since the space of vortex membranes is a coadjoint orbit, it carries a natural (Kirillov--Kostant) symplectic structure.

\begin{proposition}  \label{prop:KK=MW}
The Marsden--Weinstein symplectic structure $\omega^{MW}$ coincides with the Kirillov--Kostant symplectic structure $\omega^{KK}$.
\end{proposition}

\begin{proof} 
Suppose that $P = \partial Q$. Then $\delta_P = d \delta_Q$, where $\delta_Q$ is a singular $1$-form whose pairing with a pseudo-form $\omega$ of degree $n-1$ is given by $\int_Q \omega$. Let $v$, $w$ be divergence-free vector fields whose restrictions to the normal bundle of $P$ are two given sections $V$, $W$. Note that since the coadjoint action of $\Diff_\rho(N)$ on membranes is the natural action of diffeomorphisms on submanifolds, the corresponding infinitesimal action of $\Vect_\rho(N)$ is precisely the restriction: $\ad^*_v \delta_P = V$, $\ad^*_w \delta_P = W$. So, we have
\begin{gather}\label{KK}
\omega^{KK}_{\delta_P}(V,W) =\omega^{KK}_{\delta_P}(\ad^*_v \delta_P,\ad^*_w \delta_P) :=\langle d^{-1}\delta_P ,~[v,w] \rangle 
\\ =\langle \delta_Q, ~[v,w] \rangle  =\int_{N} \delta_Q \wedge i_{[v,w]} \rho=\int_Q i_{[v,w]} \rho
=\int_P i_vi_w\rho = \omega^{MW}_{\delta_P}(V,W)  .
\end{gather}

Note that for divergence-free vector fields $v$ and $w$ their commutator satisfies the identity $i_{[v,w]}\rho=d(i_vi_w\rho)$, which implies the second last equality. 
\end{proof}

\medskip

\subsection{Binormal equation on non-orientable manifolds}
Now, suppose that the manifold $N$ is endowed with a Riemannian metric. In this setting, there is a natural Hamiltonian flow on the symplectic manifold of vortex membranes in $N$. 
\begin{definition}
Define a Hamiltonian function on  membranes $P$ (which are co-oriented closed submanifolds of codimension 2) by taking their $(n-2)$-Riemannian volume: 
$$
H(P ) = \text{volume}(P) =  \int_P \mu_P,
$$
where $\mu_P$ is the volume pseudo-form on $P$ defined by restricting the metric from $N$. Once again,  the integral of the density $\mu_P$ is well-defined since $P$ is a co-oriented submanifold of $N$.
\end{definition}

\begin{theorem} \label{thm:binormal}
In any dimension $n\ge 3$ the Hamiltonian vector field for the Hamiltonian $H$ and the Marsden--Weinstein  symplectic $\omega^{MW}$ structure on codimension 2 membranes $P\subset N$ is 
$$
v_{H}(p)=C\cdot J({\bf MC}(p))\,,
$$ 
where $C$ is a constant depending on the geometry of $N$, $J$ is the operator of positive $\pi/2$ rotation in every (oriented) normal space $N_pP$ to $P$, and ${\bf MC}(p)$ is the mean curvature vector to $P$ at the point $p$.
\end{theorem}

Recall that the mean curvature vector  ${\bf MC}(p)\in N_pP$ for a smooth submanifold $P$ of dimension $\ell$  is the normalized trace of the second fundamental form at $p$, i.e. the trace divided by $\ell$. Equivalently, this vector ${\bf MC}(p)\in N_pP$ is the mean of curvature vectors of geodesics in $P$ passing through the point $p$ when  averaged over the sphere $S^{\ell-1}$ of all possible unit tangent vectors in $T_p P$ for these geodesics.

\begin{proof}
The above theorem holds for any Riemannian manifold $N$, oriented or not. For oriented manifolds this was obtained in \cite{KhMMJ, Shash, Vizman}. However, the consideration is local and is valid for any $N$.
The sketch of the proof is as follows. It is well known that the mean curvature vectors point in the direction of the fastest decrease of the manifold's volume, i.e. they make the field ${\bf MC}(p)=-{\rm grad }\, H(p)$. Furthermore, the Marsden--Weinstein symplectic structure is in fact the symplectic structure  in the normal plane $N_pP$ averaged over all $p\in P$. 
Hence to make the skew-gradient field $v_H(p)=-J{\rm grad }\, H(p)$ out of the gradient one, one needs to apply the operator $J$ of the almost complex structure in each normal plane, which completes the proof.
\end{proof}

\begin{definition}
The {\it binormal (or skew-mean-curvature) flow} on membranes  $P\subset N$
is given by the equation
\begin{equation}\label{skew-mean}
\partial_t P(p)=-J({\bf MC}(p))\,.
\end{equation}
\end{definition}

Note that the skew-mean-curvature flow does not stretch the submanifold $P$, while moving its points orthogonally to the mean curvatures. In particular, the volume $H(P)$ of the submanifold $P$ is preserved under this evolution, as it should, being the  Hamiltonian function of the corresponding dynamics.

For dimension $n=3$ the mean curvature vector is the curvature vector $k  \cdot{\bf n}$ of a curve $\gamma$:
${\bf MC}=k \cdot {\bf n}$, while the skew mean-curvature flow becomes the binormal equation:
$$
\partial_t \gamma=-J(k \cdot  {\bf n})=k\cdot {\bf b}=\gamma'\times \gamma'',
$$
where the last equality is valid in the arc-length parametrization $\theta$ of $\gamma$, $\gamma':=\partial \gamma/\partial \theta$.

\begin{remark}
Recall that one can lift all the objects to the naturally orientented double cover $\widetilde N$, equipped with the orientation reversing involution without fixed points. Then $P$ gives rise to its orientation cover $\widetilde P\subset \widetilde N$, and one may think of singular vorticities in the non-orientable setting as singular vorticities in the orientable setting that are equivariant under the orientation-reversing involution. This way one can reformulate the 
binormal  flow on membranes   via the equivariant version on the corresponding double cover.
\end{remark}

\end{appendices}

\bibliographystyle{plain}
\bibliography{main}

\begin{thebibliography}{10}

\bibitem{Kronrode}
G.~Adelson-Welsky and A.~Kronrode.
\newblock Sur les lignes de niveau des fonctions continues poss{\'e}dant des
  d{\'e}riv{\'e}es partielles.
\newblock {\em Comptes rendus (Doklady) de l'Acad{\'e}mie des sciences de
  l'URSS}, 49(4):235--237, 1945.

\bibitem{arnold1966geometry}
V.~Arnold.
\newblock Sur la g\'eom\'etrie diff\'erentielle des groupes de lie de dimension
  infinie et ses applications \`a l'hydrodynamique des fluides parfaits.
\newblock {\em Annales de l'Institut Fourier}, 16(1):319--361, 1966.

\bibitem{arnold1999topological}
V.~Arnold and B.~Khesin.
\newblock {\em Topological methods in hydrodynamics}, volume 125.
\newblock Springer, 1998, Second edition: 2021.

\bibitem{balabanova2022hamiltonian2}
N.~Balabanova.
\newblock A {H}amiltonian approach for point vortices on non-orientable
  surfaces {II}: the {K}lein bottle.
\newblock {\em arXiv:2202.06175}, 2022.

\bibitem{balabanova2022hamiltonian1}
N.~Balabanova and J.~Montaldi.
\newblock Hamiltonian approach for point vortices on non-orientable surfaces
  {I}: the {M}obius band.
\newblock {\em arXiv:2202.06160}, 2022.

\bibitem{bruveris2018moser}
M.~Bruveris, P.~Michor, A.~Parusi{\'n}ski, and A.~Rainer.
\newblock Moser’s theorem on manifolds with corners.
\newblock {\em Proceedings of the American Mathematical Society},
  146(11):4889--4897, 2018.

\bibitem{de2012differentiable}
G.~De~Rham.
\newblock {\em Differentiable manifolds: forms, currents, harmonic forms},
  volume 266.
\newblock Springer Science \& Business Media, 2012.

\bibitem{donnelly2021gravitational}
W.~Donnelly, L.~Freidel, S.F. Moosavian, and A.J. Speranza.
\newblock Gravitational edge modes, coadjoint orbits, and hydrodynamics.
\newblock {\em Journal of High Energy Physics}, 2021(9):1--63, 2021.

\bibitem{Vizman}
S.~Haller and C.~Vizman.
\newblock Non-linear {G}rassmannians as coadjoint orbits.
\newblock {\em Mathematische Annalen}, 329:771--785, 2003.

\bibitem{izosimov2016characterization}
A.~Izosimov and B.~Khesin.
\newblock Characterization of steady solutions to the 2{D} {E}uler equation.
\newblock {\em International Mathematics Research Notices},
  2017(24):7459--7503, 2016.

\bibitem{izosimov2017classification}
A.~Izosimov and B.~Khesin.
\newblock Classification of {C}asimirs in 2{D} hydrodynamics.
\newblock {\em Moscow Mathematical Journal}, 17(4):699--716, 2017.

\bibitem{izosimov2016coadjoint}
A.~Izosimov, B.~Khesin, and M.~Mousavi.
\newblock Coadjoint orbits of symplectic diffeomorphisms of surfaces and ideal
  hydrodynamics.
\newblock {\em Annales de l'Institut Fourier}, 66(6):2385--2433, 2016.

\bibitem{KhMMJ}
B.~Khesin.
\newblock Symplectic structures and dynamics on vortex membranes.
\newblock {\em Moscow Mathematical Journal}, 12(2):413--434, 2012.

\bibitem{kirillov2023classification}
I.~Kirillov.
\newblock Classification of coadjoint orbits for symplectomorphism groups of
  surfaces.
\newblock {\em International Mathematics Research Notices}, 2023(7):6219--6251,
  2023.

\bibitem{nicolaescu2007invitation}
L.~Nicolaescu.
\newblock {\em An invitation to Morse theory}.
\newblock Springer, 2007.

\bibitem{penna2020sdiff}
R.~Penna.
\newblock {$\mathrm{SDiff}( S^2)$} and the orbit method.
\newblock {\em Journal of Mathematical Physics}, 61:012301, 2020.

\bibitem{Reeb}
G.~Reeb.
\newblock Sur les points singuliers d\'une forme de {P}faff compl{\`e}tement
  int{\'e}grable ou d'une fonction num{\'e}rique.
\newblock {\em Comptes Rendus de l'Académie des Sciences. Paris},
  222:847--849, 1946.

\bibitem{Shash}
B.~Shashikanth.
\newblock Vortex dynamics in {$\mathbb R^4$}.
\newblock {\em Journal of Mathematical Physics}, 53:3103--3135, 2012.

\bibitem{shastri2011elements}
A.R. Shastri.
\newblock {\em Elements of differential topology}.
\newblock CRC Press, 2011.

\bibitem{vanneste2021vortex}
J.~Vanneste.
\newblock Vortex dynamics on a {M}{\"o}bius strip.
\newblock {\em Journal of Fluid Mechanics}, 923, 2021.

\end{thebibliography}
\addcontentsline{toc}{section}{References}
\end{document}